\documentclass[11pt,leqno]{amsart}

\usepackage{amsmath,amsthm}
\usepackage{latexsym}
\usepackage{amssymb}
\usepackage{mathtools}
\usepackage[parfill]{parskip}
\usepackage{mathrsfs}
\usepackage{enumerate}
\usepackage{booktabs} 
\usepackage{hyperref} 
\usepackage{xcolor} 
\usepackage{tikz-cd}
\usepackage{slashed} 

\newtheorem{theorem}{Theorem}

\newtheorem{corollary}[theorem]{Corollary}

\newtheorem{lemma}[theorem]{Lemma}
\newtheorem*{claim*}{Claim}
\newtheorem{proposition}[theorem]{Proposition}

\theoremstyle{definition}

\newtheorem{remark}[theorem]{Remark}

\newtheorem{question}[theorem]{Question}

\numberwithin{theorem}{section}

\begin{document}

\bibliographystyle{plain}

\title{Boundary Dehn twists are often commutators}
\author {Ayodeji Lindblad}
\address{Department of Mathematics, Massachusetts Institute of Technology, Cambridge, Massachusetts 02139}
\email{ayodeji@mit.edu}
\maketitle

\begin{abstract}
For $X$ any complete intersection of even complex dimension or any connected sum thereof (or, more generally, any space among certain broad classes of smooth manifolds), we concretely construct diffeomorphisms $a,c$ of punctured $X$ rel boundary whose commutator $[a,c]$ represents the smooth mapping class (rel boundary) of the boundary Dehn twist. This shows that boundary Dehn twists on 4-manifolds known to be nontrivial in the smooth mapping class group rel boundary by work of Baraglia-Konno, Kronheimer-Mrowka, J. Lin, and Tilton become trivial after abelianization, generalizing work of Y. Lin, who applied an argument based on the global Torelli theorem and an obstruction of Baraglia-Konno to prove that the abelianized boundary Dehn twist on the punctured $K3$ surface is trivial.
\end{abstract}

\section{Background and main results}
\label{sec:intro}

For a smooth oriented manifold $X$, we consider the groups $\operatorname{Diff}^{+}(X)$ of orientation-preserving diffeomorphisms of $X$ and $\pi_0(\operatorname{Diff}^{+}(X))$ of smooth isotopy classes of $\operatorname{Diff}^{+}(X)$, called the \emph{smooth mapping class group of $X$}. When $X$ is connected and closed, we denote by $X^\circ$ the complement of an open ball in $X$ and consider the groups $\operatorname{Diff}^{+}(X^\circ,\partial)$ of elements of $\operatorname{Diff}^{+}(X^\circ)$ which restrict to the identity in a neighborhood of the boundary and $\pi_0(\operatorname{Diff}^{+}(X^\circ,\partial))$ of smooth isotopy classes of $\operatorname{Diff}^{+}(X^\circ,\partial)$ relative to the boundary (henceforth \emph{rel boundary}), called the \emph{smooth mapping class group of $X^\circ$ rel boundary}. For $[\alpha]$ a generator of $\pi_1(SO(n))$, a diffeomorphism of $X^\circ$ taking $(t,\omega)\mapsto(t,\alpha(t)\omega)$ on a collar neighborhood $\nu(\partial X^\circ)\cong[0,1)\times S^{n-1}$ of $\partial X^\circ$ and restricting to the identity elsewhere has smooth mapping class $\delta_X\in\pi_0(\operatorname{Diff}^{+}(X^\circ,\partial))$ rel boundary which is referred to as \emph{the boundary Dehn twist on $X^\circ$}. 
Denoting by $\operatorname{BDiff}^{+}(X^\circ,\partial)$ the classifying space of $\operatorname{Diff}^{+}(X^\circ,\partial)$, the association 
$\pi_0(\operatorname{Diff}^{+}(X^\circ,\partial))\cong\pi_1(\operatorname{BDiff}^{+}(X^\circ,\partial))$
gives rise to the abelianization map
\[\cdot^{{ab}}:\pi_0(\operatorname{Diff}^{+}(X^\circ,\partial))\to H_1(\operatorname{BDiff}^{+}(X^\circ,\partial);\Bbb Z),\]
allowing us to consider \emph{the abelianized boundary Dehn twist $\delta_X^{ab}$ on $X^\circ$}. The study of boundary Dehn twists and their abelianizations is then motivated by the recent wave of interest in boundary Dehn twists on 4-manifolds \cite{BaragliaKonno22,BaragliaKonno25,KangParkTaniguchi26,Konno...24b,Konno...24a,KonnoMallickTaniguchi24,KrannichKupers25,KronheimerMrowka20,Lin23,OrsonPowell25,Tilton25} alongside the sustained interest in abelianizations of mapping class groups \cite{Birman70,GalatiusRandal-Williams16,Korkmaz02,Krannich20,KreckSu25,Mumford67,Powell78} and cohomology classes of $\operatorname{BDiff}^+(X)$ \cite{
GalatiusRandal-Williams14,
GalatiusRandal-Williams17,Harer85,
Konno21,KonnoLin23,
MasdenWeiss07,Miller86,Morita87,Mumford83}, which correspond to characteristic classes of smooth $X$-bundles. For a simply-connected closed 4-manifold $X$, Y. Lin asked whether the abelianized boundary Dehn twist on $X^\circ$ is trivial and observed \cite[Proposition 2.1]{Lin25} that such triviality is equivalent to the existence of a smooth orientable $X$-bundle over an orientable closed surface whose total space is not spin. Y. Lin then applied an argument based on the global Torelli theorem (via an argument \cite[Theorem 1.1]{BaragliaKonno23} of Baraglia and Konno) and an obstruction \cite[Proposition 4.21]{BaragliaKonno22} of Baraglia and Konno 
to prove existence of such a bundle when $X=K3$, answering this case of the question in the affirmative. We observed the following concrete reproof of this result: viewing the $K3$ surface as the complex hypersurface of $\Bbb{P}^3$ cut out by the polynomial $z_0^4+z_1^4+z_2^4+z_2z_3^3$, the involutions of $\Bbb{P}^3$ induced by the antipodal map in the zeroth complex coordinate and complex conjugation restrict to commuting smooth involutions $a,c\in\operatorname{Diff}^+(K3)$ with a shared fixed point $[0:0:0:1]\in K3$, at which $a$ and $c$ admit certain local behavior of interest. From such involutions, we may then construct diffeomorphisms $a^\circ,c^\circ\in\operatorname{Diff}^+(K3^\circ,\partial)$ whose commutator $[a^\circ,c^\circ]$ represents the boundary Dehn twist on $K3^\circ$. Underlying this observation is a general concrete construction of such commuting smooth involutions on all smooth complete intersections (by which we mean complex $(n-m)$-manifolds cut out in $\Bbb P^n$ by $m<n$ homogeneous polynomials) of even complex dimension, on all connected sums thereof, and on further broad classes of smooth manifolds; we formalize this construction to provide a broadly applicable concrete realization of the boundary Dehn twist as a commutator in the smooth mapping class group rel boundary. This is communicated by Theorem \ref{thm:main}:

\begin{theorem}[Main Theorem]\label{thm:main}
Let $X$ be any smooth complete intersection of even complex dimension or any connected sum thereof (or any other space among the broad classes of smooth manifolds discussed at the end of this section). We may construct $a^\circ,c^\circ\in\operatorname{Diff}^{+}(X^\circ,\partial)$ whose commutator $[a^\circ,c^\circ]$ represents the boundary Dehn twist $\delta_X$ on $X^\circ$.
\end{theorem}

Theorem \ref{thm:main} has the consequence that the abelianized boundary Dehn twist is often trivial, even on 4-manifolds on which it is known---by work of Baraglia and Konno \cite{BaragliaKonno22,BaragliaKonno25}, Kronheimer and Mrowka \cite{KronheimerMrowka20}, J. Lin \cite{Lin23}, and Tilton \cite{Tilton25}---that the boundary Dehn twist is nontrivial in the smooth mapping class group rel boundary. Specifically, for a simply-connected closed 4-manifold $X$, it is only known that the boundary Dehn twist on $X^\circ$ is nontrivial in the smooth mapping class group rel boundary (in which case it is a \emph{relatively exotic mapping class} of $X^\circ$ \cite[Theorem E]{OrsonPowell25}) for $X$ homeomorphic to the $K3$ surface \cite{BaragliaKonno22,KronheimerMrowka20}, for $X$ the once-stabilized $K3$ surface $K3\,\#\,(S^2\times S^2)$ \cite{Lin23}, for any $X$ satisfying certain conditions \cite[Theorem 1.4]{BaragliaKonno25} (principal examples of which include a large collection of complete intersections and double logarithmic transforms of the elliptic surfaces $E(4n-2)$ \cite[Theorem 1.5]{BaragliaKonno25}), and for $X$ the connected sum $K3\,\#\,K3$ of two $K3$ surfaces \cite{Tilton25}. The $K3$ surface, the once-stabilized $K3$ surface, the infinitely many complete intersections satisfying the aforementioned conditions, and the connected sum of two $K3$ surfaces are all smooth complete intersections or connected sums thereof, so Theorem \ref{thm:main} gives concrete commutator expressions for the nontrivial boundary Dehn twists on their punctures in the smooth mapping class group rel boundary, showing that they become trivial after abelianization (a fact shown in the setting of the $K3$ surface by Y. Lin \cite{Lin25}). In fact, Corollary \ref{cor:ourBKlem} in Subsection \ref{sub:Torelli}---which is verified using arguments \cite[Subsection 4.2]{BaragliaKonno25} of Baraglia and Konno---notes that the boundary Dehn twist on the puncture of any space $X$ satisfying the aforementioned conditions will even have nontrivial image in the abelianization of the Torelli subgroup of the smooth mapping class group rel boundary, so boundary Dehn twists on the punctures of complete intersections satisfying these conditions will have nontrivial image in the abelianization of the Torelli subgroup, despite Theorem \ref{thm:main} showing they have trivial image in the abelianization of the full smooth mapping class group rel boundary. 
Furthermore, note that the elliptic surfaces $E(n)$ are natural candidates for future study of boundary Dehn twists for $n$ even (as this is when $E(n)$ is spin; the boundary Dehn twist is trivial on the puncture of any non-spin simply-connected closed manifold \cite[Corollary A.5]{OrsonPowell25}). As observed in Section \ref{sec:prod}, these spaces belong to another broad class of spaces $X$ for which we verify Theorem \ref{thm:main}.

These results have connections to the study of bundles over orientable surfaces. Corollary \ref{cor:torus} communicates the notable consequence of Theorem \ref{thm:main} that smooth orientable $X$-bundles over $T^2$ with spin fibers and non-spin total spaces exist in some generality:

\begin{corollary}\label{cor:torus}
For $X$ as in Theorem \ref{thm:main}, we may construct a smooth orientable $X$-bundle over $T^2$ whose total space is not spin (specifically, this bundle admits a section whose normal bundle is not spin).
\end{corollary}

Each of the following sections presents and proves a proposition of interest in verifying Theorem \ref{thm:main}. Section \ref{sec:multi} proves Proposition \ref{pro:cpleteint}, which states that, on any smooth complete intersection of even complex dimension (in some complex projective space), we may construct commuting smooth orientation-preserving involutions $a$ and $c$ which exhibit certain behavior of interest near a common fixed point. Section \ref{sec:prod} then proves Proposition \ref{pro:multi}, which gives the corresponding statement for a more general class of smooth complete intersections of even complex dimension in products of complex projective spaces. Our proofs of these propositions imply that these smooth manifolds typically admit the structure of a holomorphic double branched cover with a compatible real structure whose branch divisor contains a real point; in fact, the deck transformation and real structure of any such double branched cover of even complex dimension also give such involutions $a,c$. For any 2-fold connected sum of a double branched cover of (real) dimension at least 3, we may arrange that the same is true of the natural extension of the deck transformation to the connected sum and the map exchanging the connected summands. By Proposition \ref{pro:mainprop} in Section \ref{sec:proproof} (which, when $\dim X>4$, we combine with Remark \ref{rmk:genprop}), any such involutions $a,c$ of a manifold $X$ give rise to an explicit commutator representative for the boundary Dehn twist on $X^\circ$ in the smooth mapping class group rel boundary, proving Theorem \ref{thm:main} when $X$ is any of the aforementioned spaces. Lemma \ref{lem:connsum} then applies to complete the proof of the theorem altogether, showing it holds when $X$ is any connected sum of such spaces by verifying that concrete commutator representatives for the boundary Dehn twists on the punctures of two manifolds of the same dimension at least 4 can be combined to give rise to such a representative for the boundary Dehn twist on the puncture of their connected sum. Following this, we also observe that the mapping torus of such involutions $a,c$ will always admit a section whose normal bundle is not spin to verify Corollary \ref{cor:torus}, then describe how arguments \cite[Subsection 4.2]{BaragliaKonno25} of Baraglia and Konno can be used to show (as is communicated by Corollary \ref{cor:ourBKlem}) that the boundary Dehn twists proved by these authors to be nontrivial in the smooth mapping class group rel boundary will even have nontrivial image in the abelianization of the Torelli subgroup, further contrasting the result of Theorem \ref{thm:main}.

We omit further examples of spaces to which Theorem \ref{thm:main} applies for brevity and invite the reader to consider settings which may be of interest. In light of the generality of the theorem, we conclude by restating the motivating question \cite{Lin25} posed by Y. Lin with a more pointed framing:

\begin{question}
Does there exist any smooth simply-connected closed 4-manifold $X$ for which the abelianized boundary Dehn twist $\delta_{X}^{ab}$ is nontrivial? 
\end{question}

\section{Complete intersections}
\label{sec:multi}

In this section, we prove Proposition \ref{pro:cpleteint}, which shows that any smooth complete intersection of even complex dimension admits diffeomorphisms $a$ and $c$ as in Proposition \ref{pro:mainprop} (where $f$ is the identity diffeomorphism, and where we appeal to Remark \ref{rmk:genprop} when the complete intersection has complex dimension greater than 2). Combined, these results then give explicit commutator representatives for the boundary Dehn twists on the punctures of all such spaces, verifying that Theorem \ref{thm:main} holds for $X$ any smooth complete intersection of even complex dimension (and, applying Lemma \ref{lem:connsum}, any connected sum thereof).

\begin{proposition}\label{pro:cpleteint}
Any smooth complete intersection $X$ admits commuting smooth involutions $a$ and $c$ with a shared fixed point $x$ which respectively act as the antipodal map in one complex coordinate and complex conjugation in suitable complex local coordinates about $x$ (so, $a$ is always orientation-preserving, while $c$ is orientation-preserving exactly when $X$ has even complex dimension).
\end{proposition}

We first introduce complete intersections (by which we specifically mean positive-dimensional complete intersections in complex projective spaces) in Subsection \ref{sub:defex}, then justify why smooth complete intersections admit certain especially symmetric presentations in Subsection \ref{sub:cutting}, and finally use the symmetries of these presentations to define the involutions whose existence Proposition \ref{pro:cpleteint} communicates in Subsection \ref{sub:diffeos}.

\subsection{Definition and examples}
\label{sub:defex}

Consider the $n$-(complex)-dimensional complex projective space $\Bbb P^n$ alongside $m<n$ homogeneous polynomials 
\[p:=(p_1,\dots,p_{m}):\Bbb C^{n+1}\to\Bbb C^m\]
of degrees $d:=(d_1,\dots,d_{m})$, so we have
\[p_i(\lambda z)=\lambda^{d_i}p_i(z)\]
for any $z\in\Bbb C^{n+1}$ and $\lambda\in\Bbb C$. The set
\[X:=\{[z]\in\Bbb P^n\:|\:p(z)=0\}\]
is well-defined and called \emph{the subset of $\Bbb P^n$ cut out by $p$}. We say a point $x\in X$ is \emph{singular} when the differential of the map from homogeneous coordinates on $\Bbb P^n$ about $x$ to $\Bbb C^m$ induced by $p$ fails to be surjective at $x$ and that $X$ is a \emph{smooth complete intersection} when $X$ has no singular points, in which case the implicit function theorem shows $X$ will be a complex manifold of dimension $n-m$. We assume that any subset $X$ of $\Bbb P^n$ referred to as a smooth complete intersection has positive dimension $n-m\geq1$. The Lefschetz hyperplane theorem \cite[Theorem 7.4]{Milnor63} gives that any such $X$ will be connected and, when $n-m\geq2$, simply-connected. Moreover, by an observation of Thom, a generic choice of homogeneous polynomials of degrees $d$ will cut out a smooth complete intersection in $\Bbb P^{n}$ diffeomorphic to $X$. 

For any positive integers $n$ and $d$, we can see that $z_0^d+\cdots+z_{n-1}^d+z_{n-1}z_{n}^{d-1}$ cuts out a smooth complex hypersurface in $\Bbb P^n$, which we denote by $X_d$ when $n=3$. We then have that the underlying smooth manifold of $X_{2}$ is $S^2\times S^2$, and that of $X_{4}$ is the $K3$ surface. In addition, the infinite family of smooth manifolds underlying the complex hypersurfaces $X_{32m+4}$ for $m$ a nonnegative integer congruent to 0 or 1 modulo 3 are among those whose punctures' boundary Dehn twists Baraglia and Konno showed \cite[Theorem 1.5]{BaragliaKonno25} are nontrivial in the smooth mapping class group rel boundary, as we can see \cite[Exercises 1.3.13(c)(e)]{GompfStipsicz99} that a smooth complete intersection $X$ cut out by $n-2$ polynomials of degrees $d_1,\dots,d_{n-2}$ in $\Bbb{P}^n$ has first Chern class 
$n+1-d_1-\cdots-d_{n-2}$
times a primitive class and signature
\[\frac13(n+1-d_1^2-\cdots-d_{n-2}^2)d_1\cdots d_{n-2}.\]
The polynomials cutting out these complex hypersurfaces exhibit symmetries (for $d$ even) which we show can be much more broadly realized in Subsection \ref{sub:cutting}.

\subsection{Symmetric homogeneous cutting polynomials}
\label{sub:cutting}

We now prove Lemma \ref{lem:cuttingpolys}, which shows that any smooth complete intersection is diffeomorphic to a complete intersection cut out by polynomials exhibiting certain symmetries used in Lemma \ref{lem:cpletediffeo} to define the involutions whose existence Proposition \ref{pro:cpleteint} communicates.

\begin{lemma}\label{lem:cuttingpolys}
For any positive integer $n$ and finite tuple $d$ of positive integers, consider a smooth complete intersection $X$ cut out in $\Bbb P^n$ by polynomials of degrees $d$. A generic choice of homogeneous polynomials of degrees $d$ with only real coefficients whose constituent monomials each have even power of $z_0$ and positive power of some coordinate function other than $z_n$ will cut out a complete intersection in $\Bbb P^n$ diffeomorphic to $X$.
\end{lemma}

\begin{proof}
With notation as in the lemma statement, we write $d=:(d_1,\dots,d_m)$, so $d_i$ is the degree of the $i$th polynomial cutting out $X$ and $m<n$, and set
$Q:=Q_1\times\cdots\times Q_{m}$
for $Q_i$ the complex vector space of homogeneous polynomials of degree $d_i$ on $\Bbb C^{n+1}$ whose constituent monomials each have even power of $z_0$ and positive power of some coordinate function other than $z_n$. For a subset $S$ of $\Bbb P^n$ and $\Delta$ the map taking an $m\times(n+1)$ matrix to the $\binom{n+1}{m}$-tuple of determinants of its $m\times m$ minors, we denote by
\begin{gather*}
\mathscr D(S):=\{(q,[s])\in Q\times S\:|\:(q(s),\Delta(dq_s))=0\}
\end{gather*}
the subset of $Q\times S$ cut out by the homogeneous polynomials 
\[(q,z)\mapsto(q(z),\Delta(dq_z))\]
on $Q\times\Bbb C^{n+1}\ni(q,z)$, so $(q,[s])\in\mathscr D(S)$ exactly when $[s]\in S$ is a singular point of the subset of $\Bbb P^n$ cut out by $q\in Q$. Considering the projection
$\mathscr Q:Q\times \Bbb P^n\to Q$,
the set $D(S):=\mathscr Q(\mathscr D(S))$ is then exactly the set of $q\in Q$ which cut out a subset of $\Bbb P^n$ which has singular points in $S$. When $S$ is Zariski-closed in $\Bbb P^n$, $Q\times S$ will be Zariski-closed in $Q\times\Bbb P^n$, so $\mathscr D(S)$ will be Zariski-closed in $Q\times \Bbb P^n$. $\mathscr Q$ is proper, so $D(S)$ will be Zariski-closed in the complex vector space $Q$, and will thus be cut out in $Q$ (whose ring of polynomials we note is Noetherian by Hilbert's basis theorem \cite[Theorem 7.5]{AtiyahMacdonald69}) by finitely many polynomials. Then, writing $D:=D(\Bbb P^n)$ and denoting by $R$ the real vector space of elements of $Q$ with only real coefficients, $D\cap R$ is cut out in $R$ by the real and imaginary parts of the finitely many polynomials which cut out $D$ in $Q$. Any polynomials which vanish on the entirety of $R$ must vanish on the entirety of $Q$, meaning that $D\cap R$ will be a proper subset of $R$ if $D$ is a proper subset of $Q$. In this case, because $D\cap R$ will be cut out by finitely many polynomials which do not all vanish on all of $R$, a generic choice of $r\in R$ will lie outside of $D$. Recalling that $r\in R$ lies outside of $D$ exactly when $r$ cuts out a smooth complete intersection in $\Bbb P^n$ and that such a smooth complete intersection will be diffeomorphic to $X$ by the observation of Thom that any two smooth complete intersections cut out in $\Bbb P^n$ by polynomials of the same degrees are diffeomorphic (in fact, even smoothly ambiently isotopic; this follows from an argument similar to that which we are now exhibiting), the proof will be complete if we show that $D$ is a proper subset of $Q$. Noting that a union of two proper Zariski-closed subsets of $Q$ will certainly still be a proper subset of $Q$, we set
\begin{gather*}
A:=\{[a]\in\Bbb P^n\:|\:a_{1}=\cdots=a_{n-1}=0\},\quad B:=\Bbb P^n\setminus A,
\end{gather*}
so we have $D=D(A)\cup D(B)$. We may thus complete the proof by showing that the Zariski-closures of $D(A)$ and $D(B)$ are proper subsets of $Q$.

As $A$ is a Zariski-closed subset of $\Bbb P^n$, the argument presented above shows that $D(A)$ will be a Zariski-closed subset of $Q$. Therefore, we need only show that $D(A)$ is a proper subset of $Q$. Writing $q_A:=(q_{A,1},\dots,q_{A,m})$ for
\[q_{A,i}(z):=\begin{cases}
z_i, & d_i=1 \\
z_0^{d_i-1}z_i+z_0^{d_i-1}z_n+z_iz_n^{d_i-1} & d_i>1\text{ odd} \\
z_0^{d_{i}}+z_iz_n^{d_{i}-1}, & d_i\text{ even}
\end{cases},\]
we can see that $q_{A,i}\in Q_i$ (where, when $d_i$ is odd, we note that in light of the standing assumption that smooth complete intersections are positive-dimensional we initiated when presenting their definition, we have $m<n$), so $q_A\in Q$. Additionally, denoting by $\delta_{i,j}$ the Kronecker delta (which equals 1 when $i=j$ and 0 otherwise), we can then directly compute in each case that at any point $[a]\in A$ such that $q_{A,i}(a)=0$, we have
\[\begin{bmatrix}
\frac{\partial q_{A,i}}{\partial z_1} & \cdots & \frac{\partial q_{A,i}}{\partial z_m}
\end{bmatrix}(a)=\begin{bmatrix}
\delta_{i,1} & \cdots & \delta_{i,m}
\end{bmatrix}.\]
Therefore, for any $[a]\in A$ at least one of the conditions
\[q_A(a)\neq0,\quad\Delta(d(q_A)_a)\ni\det\begin{bmatrix}
\frac{\partial q_A}{\partial z_1} & \cdots & \frac{\partial q_A}{\partial z_m}
\end{bmatrix}(a)=1\neq0\]
is satisfied, so $q_A\not\in D(A)$. This completes our proof that the Zariski-closure of $D(A)$ (which coincides with $D(A)$) is a proper subset of $Q$.

We now complete the proof by showing that the Zariski-closure of $D(B)$ is a proper subset of $Q$. To this end, pick $[b]\in B$, so there exists $j\in\{1,\dots,n-1\}$ such that $b_j\neq0$. Then, for
\[q_{b,i}(z):=(\delta_{i,1}z_j^{d_1},\dots,\delta_{i,m}z_j^{d_m})\in Q,\]
we see that the set
$\{q_{b,1}(b),\dots,q_{b,m}(b)\}$
spans $\Bbb C^m$. Therefore, the kernel of the linear evaluation map
\[\textrm{ev}_{[b]}:Q\to\Bbb C^m,\quad q\mapsto q(b/b_j)\]
has dimension $\dim Q-m$. Denoting by $B_j$ the set of all $[b]\in B$ satisfying $b_j\neq0$, we then have a morphism
\[Q\times B_j\to\Bbb C^m\times B_j,\quad(q,[b])\mapsto(q(b/b_j),[b])\]
of (trivial) algebraic vector bundles over the Zariski-open subset $B_j$ of $\Bbb P^n$. We then consider the kernel bundles $\mathscr E_j\to B_j$ of these morphisms for $j\in\{1,\dots,n-1\}$, which we can observe give rise to an algebraic vector subbundle $\mathscr E\to B$ of the trivial bundle $Q\times B\to B$ over the Zariski-open subset $B$ of $\Bbb P^{n}$. Notably, we may conclude that the total space $\mathscr E$ of this bundle is a smooth subvariety of $Q\times B$, and thus that the projection  $\mathscr Q|_{\mathscr E}:\mathscr E\to \overline E$ is a morphism of varieties over $\Bbb C$ with nonsingular domain (in the language of the book of Hartshorne \cite[Remark II.4.10.1]{Hartshorne77}), where $\overline E$ is the Zariski-closure of $E:=\mathscr Q(\mathscr E)$ in $Q$. Therefore, we can apply the generic smoothness theorem \cite[Corollary III.10.7]{Hartshorne77} to see that $\mathscr Q|_{\mathscr E}$ will be smooth on $\mathscr Q|_{\mathscr E}^{-1}(\overline E\setminus F)$ for a proper Zariski-closed subset $F$ of $\overline E$. As the fiber $\mathscr Q|_{\mathscr E}^{-1}(q)$ is exactly $\{q\}$ times the intersection of the subset of $\Bbb P^n$ cut out by $q$ with $B$, $D(B)$ lies in the proper Zariski-closed subset $F$ of $Q$, so the Zariski-closure of $D(B)$ is a proper subset of $Q$. In light of our above proof that the Zariski-closure of $D(A)$ is a proper subset of $Q$ and the fact that $D=D(A)\cup D(B)$, this completes the proof that $D$ is a proper subset of $Q$, which we observed above completes the proof altogether.
\end{proof}

\subsection{Involutions of complete intersections}
\label{sub:diffeos}

Lemma \ref{lem:cpletediffeo} states that any smooth complete intersection cut out by polynomials exhibiting the symmetries of interest in Lemma \ref{lem:cuttingpolys} admits involutions as in Proposition \ref{pro:cpleteint}:

\begin{lemma}\label{lem:cpletediffeo}
For any positive integer $n$ and any smooth complete intersection $X$ cut out in $\Bbb P^n$ by polynomials with only real coefficients whose constituent monomials each have even power of $z_0$ and positive power of some coordinate function other than $z_n$, the maps
\begin{equation}\label{eq:CPnmaps}
[z]\mapsto [-z_0:z_1:\cdots:z_n],\quad [z]\mapsto[\overline z]
\end{equation}
on $\Bbb P^n$ respectively restrict to commuting smooth involutions $a$ and $c$ of $X$ which act as the antipodal map in one complex coordinate and complex conjugation in suitable complex local coordinates about their shared fixed point $[0:\cdots:0:1]\in X$.
\end{lemma}

\begin{proof}
With notation as in the lemma statement, say $X$ is cut out by $m$ polynomials $p$. We have that $p(0,\dots,0,1)=0$, so $x:=[0:\cdots:0:1]\in X$. We may consider complex coordinates
\[\psi:z\mapsto[z_0:\cdots:z_{n-1}:1]\]
in a neighborhood of $x$ on $\Bbb{P}^{n}$; the differential of 
\[p_\psi:z\mapsto p(z_0,\dots,z_{n-1},1),\]
at $0\in\Bbb C^n$ is then surjective because $X$ is smooth. Note also that $\partial_{0}p_\psi(0)=0$, as the $z_0$-coordinate of $x$ is zero and the polynomials $p_\psi$ admit decompositions into summand monomials whose powers of $z_{0}$ are all even (and are, notably, thus not 1). We may therefore consider a permutation $\sigma:\Bbb C^{n}\to\Bbb C^{n}$ of the coordinates of $\Bbb C^{n}$ fixing the $z_0$-coordinate such that the vectors $((\partial_{i}(p_\psi\circ\sigma))(0))_{i=n-m}^{n-1}$ are linearly independent. The implicit function theorem then gives that there exists a unique holomorphic map $P_\psi$ defined near 0 such that 
\[(p_\psi\circ\sigma)(z_{0},\dots,z_{n-m-1},P_\psi(z_{0},\dots,z_{n-m-1}))=0,\]
giving complex coordinates 
\[\psi_X:(z_{0},\dots,z_{n-m-1})\mapsto(\psi\circ\sigma)(z_{0},\dots,z_{n-m-1},P_\psi(z_{0},\dots,z_{n-m-1}))\]
about $x$ in $X$. Now, the commuting smooth involutions \eqref{eq:CPnmaps} restrict to commuting smooth involutions $a$ and $c$ of $X$ because the polynomials $p$ cutting out $X$ admit decompositions into summand monomials whose powers of $z_{0}$ are all even and because the polynomials have only real coefficients. Then,
\begin{equation*}
	\begin{gathered}
(\psi_X^{-1}\circ a\circ\psi_X)(z_0,\dots,z_{n-m-1})=(-z_0,z_1,\dots,z_{n-m-1}), \\
(\psi_X^{-1}\circ c\circ\psi_X)(z_0,\dots,z_{n-m-1})=(\overline z_0,\dots,\overline z_{n-m-1}),
\end{gathered}
\end{equation*}
exactly as desired.
\end{proof}

With $X$ as in Lemma \ref{lem:cpletediffeo}, observe that the fixed points of $a$ are exactly 
\[X\cap(\{[z]\in\Bbb P^n\:|\:z_0=0\}\cup\{[1:0:\cdots:0]\}),\]
so when $\dim_\Bbb C X\geq2$, the quotient map taking $X$ to $X/a$ will be a holomorphic double branched cover with deck transformation $a$ having compatible real structure $c$ (whose branch divisor then contains the real point $[0:\cdots:0:1]$) exactly when $[1:0:\cdots:0]\not\in X$. Noting that we will have $[1:0:\cdots:0]\not\in X$ exactly when any of the polynomials cutting out $X$ has nonzero coefficient of the monomial with positive power of only $z_0$, we will necessarily have $[1:0:\cdots:0]\in X$---and thus that $X/a$ is singular---whenever all degrees of the polynomials cutting out $X$ are odd. Otherwise, a generic choice of polynomials as in Lemma \ref{lem:cuttingpolys} will cut out a smooth complete intersection in $\Bbb P^n$ which does not contain $[1:0:\cdots:0]$; this complete intersection can then be given the structure of a holomorphic double branched cover with compatible real structure whose branch divisor contains a real point.

As was mentioned in Subsection \ref{sub:defex}, the polynomials described in that subsection are all of the form of interest in Lemma \ref{lem:cpletediffeo}, so we get explicit descriptions of how involutions of the complete intersections they cut out which satisfy the conditions of interest in the lemma arise. For example, in the case that $X=X_{2}$, these smooth involutions (under some identification $X_{2}\cong S^2\times S^2\subset(\Bbb C\times\Bbb R)^2$) take
\[(\eta_0,\xi_0,\eta_1,\xi_1)\mapsto(-\eta_0,\xi_0,\eta_1,\xi_1),\quad(\eta_0,\xi_0,\eta_1,\xi_1)\mapsto(\overline\eta_0,\xi_0,\overline\eta_1,\xi_1).\]
We also realize the smooth involutions produced on the $K3$ surface $X_{4}$ from another perspective which may be of interest; we formalize this and a generalization thereof which covers the setting of the simply-connected elliptic surfaces $E(n)$ in Subsection \ref{sub:involmulti}.

In addition to exhibiting involutions as in Lemma \ref{lem:cpletediffeo} on the smooth complete intersections described in Subsection \ref{sub:defex}, Lemma \ref{lem:cpletediffeo} can be combined with Lemma \ref{lem:cuttingpolys} to directly give the result of Proposition \ref{pro:cpleteint}, whose proof was the main purpose of this section.

\section{Complete intersections in products of complex projective spaces}
\label{sec:prod}

In this section, we prove Proposition \ref{pro:multi}, which shows that any smooth complete intersection of even complex dimension cut out in some product 
\[\Bbb P^{\textbf{\textit{n}}}:=\Bbb P^{n_0}\times\cdots\times\Bbb P^{n_{\nu}}\quad(\textbf{\textit{n}}:=(n_0,\dots,n_\nu))\]
of complex projective spaces by polynomials whose multidegrees satisfy a certain constraint admits diffeomorphisms $a$ and $c$ as in Proposition \ref{pro:mainprop} (where $f$ is the identity diffeomorphism, and where we appeal to Remark \ref{rmk:genprop} when the complete intersection has complex dimension greater than 2). Combined, these results then give explicit commutator representatives for the boundary Dehn twists on the punctures of all such spaces, verifying that Theorem \ref{thm:main} holds for $X$ any such space (and, applying Lemma \ref{lem:connsum}, any connected sum thereof).

\begin{proposition}\label{pro:multi}
Consider a tuple ${\textbf{\textit{n}}}:=(n_0,\dots,n_\nu)$ of positive integers, a positive integer $m<n:=n_0+\cdots+n_\nu$, and an $m\times (\nu+1)$ matrix $d$ of positive integers which satisfies
\begin{equation}\label{eq:proparities}
m_i=0\quad\text{or}\quad m-2m_i\geq n- 2n_i
\end{equation}
for some $i\in\{0,\dots,\nu\}$, where $m_i<n_i$ is the number of odd entries of the $i$th column of $d$. Any smooth complete intersection $X$ cut out in the product $\Bbb P^{\textbf{\textit{n}}}$ of complex projective spaces by polynomials of multidegrees $d$ admits commuting smooth involutions $a$ and $c$ with a shared fixed point $x\in X$ which respectively act as the antipodal map in one complex coordinate and complex conjugation in suitable complex local coordinates about $x$ (so, $a$ is always orientation-preserving, while $c$ is orientation-preserving exactly when $X$ has even complex dimension).
\end{proposition}

We first introduce complete intersections in products of complex projective spaces in Subsection \ref{sub:multi-defex}, then present the fact that these spaces admit certain especially symmetric presentations analogous to those described in Subsection \ref{sub:cutting} when they satisfy the constraint \eqref{eq:proparities} in Subsection \ref{sub:symmulti}, and finally use the symmetries of these presentations to define the involutions analogous to those described in Subsection \ref{sub:diffeos} whose existence Proposition \ref{pro:multi} communicates in Subsection \ref{sub:involmulti}.

\subsection{Definition and examples}
\label{sub:multi-defex}

Consider a tuple $\textbf{\textit{n}}:=(n_0,\dots,n_\nu)$ alongside $m<n:=n_0+\cdots+n_\nu$ multihomogeneous polynomials 
\[p:=(p_1,\dots,p_m):\Bbb C^{n_0+1}\times\cdots\times\Bbb C^{n_\nu+1}\to\Bbb C^m\]
of multidegrees $d$, so we have
\[p_i(\lambda_0z_0,\dots,\lambda_\nu z_\nu)=\lambda_0^{d_{i,0}}\cdots\lambda_\nu^{d_{i,\nu}}p_i(z_0,\dots,z_\nu)\]
for all $z_j\in\Bbb C^{n_j+1}$ and $\lambda_j\in\Bbb C$. The set
\[X:=\{([z_0],\dots,[z_\nu])\in\Bbb P^{\textbf{\textit{n}}}\:|\:p(z_0,\dots,z_\nu)=0\}\]
is well-defined and called \emph{the subset of $\Bbb P^{\textbf{\textit{n}}}$ cut out by $p$}. We say a point $x\in X$ is \emph{singular} when the differential of the map from multihomogeneous coordinates on $\Bbb P^{\textbf{\textit{n}}}$ about $x$ to $\Bbb C^m$ induced by $p$ fails to be surjective at $x$ and that $X$ is a \emph{smooth complete intersection in a product of complex projective spaces} when it has no singular points, in which case it will be a complex manifold of dimension $n-m$. We assume any subset $X$ of $\Bbb P^{\textbf{\textit{n}}}$ referred to as a smooth complete intersection in a product of complex projective spaces has positive dimension $n-m\geq1$. We also restrict ourselves to the case when all multidegrees cutting out such $X$ have no zero entries: in this setting, the Lefschetz hyperplane theorem \cite[Theorem 7.4]{Milnor63} gives that $X$ will be connected and, when $n-m\geq2$, simply-connected. Moreover, the observation of Thom applies as in Subsection \ref{sub:defex} to show that if $X$ is a smooth complete intersection cut out in $\Bbb P^{\textbf{\textit{n}}}$ by polynomials of multidegrees $d$, a generic choice of multihomogeneous polynomials of multidegrees $d$ will cut out a complete intersection in $\Bbb P^{\textbf{\textit{n}}}$ diffeomorphic to $X$.

Denote by $X_{2,m,n}$ the complex hypersurface cut out by the polynomial
\[z_{0,0}^2p_1(z_{1,0},z_{1,1},z_{2,0},z_{2,1})+z_{0,1}^2p_2(z_{1,0},z_{1,1},z_{2,0},z_{2,1})\]
in $\Bbb P^1\times\Bbb P^1\times\Bbb P^1$, for polynomials $p_1$ and $p_2$ on $\Bbb C^2\times\Bbb C^2$ of multidegrees $(m,n)$ which satisfy the generic condition that they cut out transverse smooth complex hypersurfaces in $\Bbb P^1\times\Bbb P^1$ (for example, taking 
\begin{equation}\label{eq:p1p2}
\begin{gathered}
p_1(z_{1,0},z_{1,1},z_{2,0},z_{2,1}):=z_{1,0}^mz_{2,0}^n+2z_{1,1}^mz_{2,0}^n+3z_{1,0}^mz_{2,1}^n+\pi z_{1,1}^mz_{2,1}^n, \\
p_2(z_{1,0},z_{1,1},z_{2,0},z_{2,1}):=z_{1,0}^mz_{2,0}^n+z_{1,1}^mz_{2,0}^n+z_{1,0}^mz_{2,1}^n+z_{1,0}z_{1,1}^{m-1}z_{2,1}^n,
\end{gathered}
\end{equation}
this will be true). The smooth 4-manifolds $X(m,n)$ underlying these complex surfaces make up an important family including the unique (up to diffeomorphism) simply-connected relatively minimal elliptic surface $E(n)$ with no multiple fibers and holomorphic Euler characteristic $n$ for $m=2$ (and thus including the $K3$ surface for $m=n=2$) \cite[Theorem 7.3.3]{GompfStipsicz99}. The projection
$X_{2,m,n}\to\Bbb P^1\times\Bbb P^1$
forgetting the first factor realizes $X_{2,m,n}$ as the minimal resolution of the singular double branched cover of $\Bbb P^1\times\Bbb P^1$ with branch divisor $B$ the union of the transverse smooth complex hypersurfaces cut out by $p_1$ and $p_2$. Denoting by $\Sigma_g$ the surface of genus $g$, $\iota:=\iota_g$ the hyperelliptic involution of $\Sigma_g$ (an involution with $2g+2$ fixed points $\textrm{Fix}(\iota)$) as in Figure \ref{fig:Hyperelliptic}, and $\pi:\Sigma_g\to\Bbb P^1\cong\Sigma_g/\iota$ the projection, $B$ has the same bidegree $(2m,2n)$ as the divisor
\[B':=(\pi(\text{Fix}(\iota_{m-1}))\times\Bbb{P}^1)\cup(\Bbb{P}^1\times\pi(\textrm{Fix}(\iota_{n-1}))).\]
The divisors $B$ and $B'$ are each reduced and nodal, so work \cite{Atiyah58} of Atiyah shows that the minimal resolutions of the singular double branched covers branched along each divisor will be diffeomorphic to some double branched covers branched along smooth divisors of bidegree $(2m,2n)$, which will in turn be diffeomorphic to one another by path-connectedness of the space of smooth divisors of a given bidegree 
(as can be seen using Bertini's theorem \cite[Corollary III.10.9]{Hartshorne77}) combined with Ehresmann's fibration theorem \cite{Ehresmann51}. Thus, $X(m,n)$ can also be realized as the minimal resolution of the singular double branched cover of $\Bbb P^1\times\Bbb P^1$ with branch divisor $B'$. We can then realize this desingularized double branched cover as the quotient by $\iota\times\iota$ of the blow-up of $\Sigma_{m-1}\times\Sigma_{n-1}$ at the fixed points of $\iota\times\iota$ using the map from this quotient to $\Bbb P^1\times\Bbb P^1$ induced by the map 
\[\Sigma_{m-1}\times\Sigma_{n-1}\to(\Sigma_{m-1}/\iota)\times(\Sigma_{n-1}/\iota)\cong\Bbb P^1\times\Bbb P^1.\]
We describe a pair of involutions of $X(m,n)$ of the form described in Proposition \ref{pro:mainprop} from this perspective after developing a general framework for the production of such pairs on smooth complete intersections in products of complex projective spaces in Subsection \ref{sub:involmulti}.

\begin{figure}
\begin{center}
\includegraphics[width=.55\textwidth]
{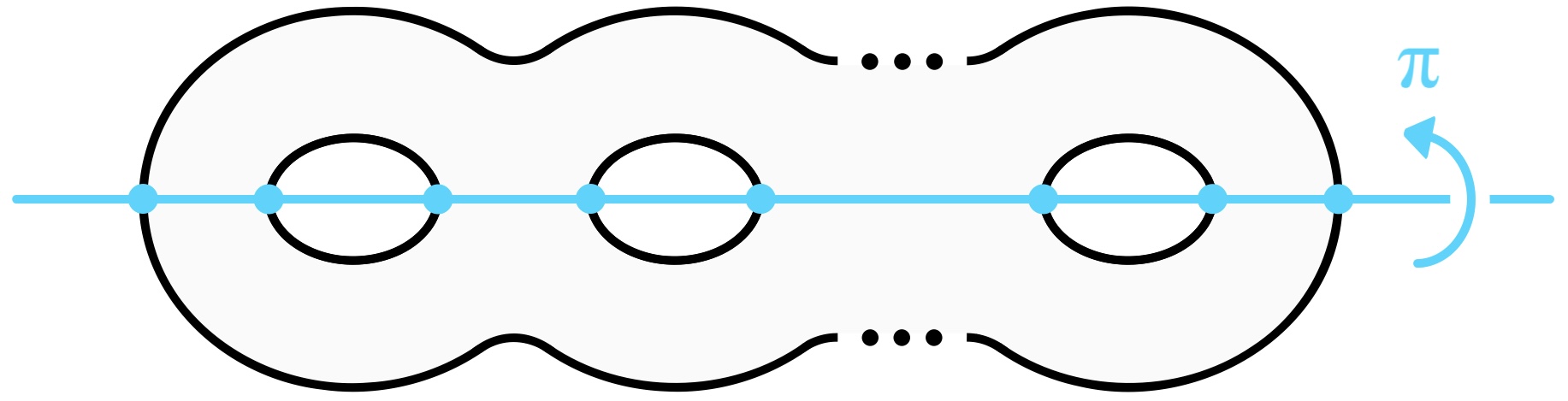}
\caption{\label{fig:Hyperelliptic}The hyperelliptic involution $\iota$ on the surface $\Sigma_g$ of genus $g$, with its $2g+2$ fixed points marked. This involution is central to one realization of the smooth simply-connected 4-manifolds $X(m,n)$ described in Subsection \ref{sub:multi-defex}}
\end{center}
\end{figure}

Taking $p_1$ and $p_2$ as in \eqref{eq:p1p2}, the polynomials cutting out the complex hypersurfaces $X_{2,m,n}$ exhibit symmetries which we show can be much more broadly realized in Subsection \ref{sub:symmulti}.

\subsection{Symmetric multihomogeneous cutting polynomials}
\label{sub:symmulti}

In this subsection, we prove Lemma \ref{lem:multipolys}, which shows that any smooth complete intersection cut out in a product of complex projective spaces by polynomials whose multidegrees satisfy a certain constraint is diffeomorphic to a complete intersection cut out in that product of complex projective spaces by polynomials exhibiting certain symmetries used in Subsection \ref{sub:involmulti} to define the involutions whose existence Proposition \ref{pro:multi} communicates.

\begin{lemma}\label{lem:multipolys}
Consider a tuple ${\textbf{\textit{n}}}:=(n_0,\dots,n_\nu)$ of positive integers, a positive integer $m<n:=n_0+\cdots+n_\nu$, and an $m\times (\nu+1)$ matrix $d$ of positive integers satisfying
\begin{equation}\label{eq:parities}
m_0=0\quad\text{or}\quad m-2m_0\geq n-2n_0,
\end{equation}
where $m_0<n_0$ is the number of odd entries of the 0th column of $d$. For a smooth complete intersection $X$ cut out in the product $\Bbb P^{\textbf{\textit{n}}}$ of complex projective spaces by polynomials of multidegrees $d$, a generic choice of multihomogeneous polynomials of multidegrees $d$ with only real coefficients whose constituent monomials each have even power of $z_{0,0}$ and positive power of some coordinate function other than any of $z_{0,n_0},\dots,z_{\nu,n_\nu}$ will cut out a complete intersection in $\Bbb P^{\textbf{\textit{n}}}$ diffeomorphic to $X$.
\end{lemma}

\begin{proof}
With notation as in the lemma statement, we write $d=:(d_1,\dots,d_m)$ for $d_i:=(d_{i,0},\dots,d_{i,\nu})$, so $d_i$ is the multidegree of the $i$th polynomial cutting out $X$. We then set
$Q:=Q_1\times\cdots\times Q_m$
for $Q_i$ the complex vector space of multihomogeneous polynomials of multidegree $d_i$ on 
$\Bbb C^{n_0+1}\times\cdots\times\Bbb C^{n_\nu+1}$
whose constituent monomials each have even power of $z_{0,0}$ and positive power of some coordinate function other than any of $z_{0,n_0},\dots,z_{\nu,n_\nu}$. For a subset $S$ of $\Bbb P^{\textbf{\textit{n}}}$, we denote by $D(S)$ the set of $q\in Q$ such that the subset of $\Bbb P^{\textbf{\textit{n}}}$ cut out by $q$ has a singular point in $S$. Arguments identical to those applied in the proof of Lemma \ref{lem:cuttingpolys} show that we may complete the proof by showing that $D:=D({\Bbb P^{\textbf{\textit{n}}}})$ is a proper subset of $Q$. 
We then set
\begin{gather*}
A:=\{[a_0]\in\Bbb P^{n_0}\:|\:a_{0,1}=\cdots=a_{0,n_0-1}=0\}\times\{\alpha'\},\quad
B:=\Bbb P^{\textbf{\textit{n}}}\setminus (A\cup C),\\
C:=\{[1:0:\cdots:0]\in\Bbb P^{n_0}\}\times(\Bbb P^{\textbf{\textit{n}}'}\setminus\{\alpha'\}) \\
(\Bbb P^{\textbf{\textit{n}}'}:=\Bbb P^{n_1}\times\cdots\times\Bbb P^{n_\nu}\ni([0:\cdots:0:1],\dots,[0:\cdots:0:1])=:\alpha'),
\end{gather*}
so we have $D=D(A)\cup D(B)\cup D(C)$. Arguments nearly identical to those presented in the proof of Lemma \ref{lem:cuttingpolys} can be used to show that the Zariski-closures of $D(A)$ and $D(B)$ are proper subsets of $Q$, so we may complete the proof by showing that $D(C)$ has empty Zariski-interior in $Q$ (in fact, it can then be straightforwardly shown that the Zariski-closure of $D(C)$ is a proper subset of $Q$, but we need not formalize this).

We first assume the left-hand condition of \eqref{eq:parities} holds, so $d_{i,0}$ is even for all $i\in\{1,\dots,m\}$. For $[c]^{\textbf{\textit{n}}}\in C$, we may select $\beta_\alpha\in\{0,\dots,n_\alpha\}$ satisfying $c_{\alpha,\beta_\alpha}\neq0$ for each $\alpha\in\{1,\dots,\nu\}$. Taking
\[q_{c,i}:=(q_{c,i,1},\dots,q_{c,i,m})\quad\text{for}\quad q_{c,i,j}(z):=\delta_{i,j}z_{0,0}^{d_{j,0}}z_{1,\beta_1}^{d_{j,1}}\cdots z_{\nu,\beta_\nu}^{d_{j,\nu}},\]
we have that $q_{c,i}\in Q$ and that $\{q_{c,1}(c),\dots,q_{c,m}(c)\}$ spans $\Bbb C^m$. The proof that the Zariski-closure of $D(C)$ is a proper subset of $Q$ when the left-hand condition of \eqref{eq:parities} holds then follows identically to the proof that the Zariski-closure of $D(B)$ is a proper subset of $Q$ in the proof of Lemma \ref{lem:cuttingpolys}.

Now, assume the right-hand condition of \eqref{eq:parities} holds. For convenience of notation, we reorder the rows of the $m\times(\nu+1)$ matrix $d=(d_{i,j})$ so that $d_{1,0},\dots,d_{m_0,0}$ are each odd and $d_{m_0+1,0},\dots,d_{m,0}$ are each even and write
\[[(z_0,\dots,z_{\nu})]^{\textbf{\textit{n}}}:=([z_0],\dots,[z_\nu])\in\Bbb P^{\textbf{\textit{n}}}\] 
for $z_i\in\Bbb C^{n_i+1}$. Considering
\[q'\in Q':=Q_1\times\cdots\times Q_{m_0},\quad q''\in Q'':=Q_{m_0+1}\times\cdots\times Q_m,\]
the subset of $\Bbb P^{\textbf{\textit{n}}}$ cut out by $(q',q'')$ will not be singular at $[c]^{\textbf{\textit{n}}}$ exactly when $d(q',q'')_{c}$ has full rank, which is certainly guaranteed when the matrices
\begin{gather}\label{eq:lmatrix}
\begin{bmatrix}
\frac{\partial q'}{\partial z_{0,1}} & \cdots & \frac{\partial q'}{\partial z_{0,n_0}}
\end{bmatrix}(c), \\
\begin{bmatrix}\label{eq:rmatrix}
\frac{\partial q''}{\partial z_{1,0}} & \cdots & \frac{\partial q''}{\partial z_{1,n_1}} & \cdots & \frac{\partial q''}{\partial z_{\nu,0}} & \cdots & \frac{\partial q''}{\partial z_{\nu,n_{\nu}}}
\end{bmatrix}(c)
\end{gather}
have ranks $m_0$ and $m-m_0$ respectively. 
By the same arguments we applied above to complete the proof of the lemma when the left-hand condition of \eqref{eq:parities} holds, all $q''\in Q''$ outside of a proper Zariski-closed subset $D''$ of $Q''$ will cut out a smooth complete intersection $\overline X_{q''}$ in the Zariski-closure
\[\overline C:=\{[1:0:\cdots:0]\in\Bbb P^{n_0}\}\times\Bbb P^{\textbf{\textit{n}}'}\] 
of $C$ in $\Bbb P^{\textbf{\textit{n}}}$ whose intersection $\overline X_{q''}\cap C$ with $C$ we denote by $X_{q''}$, so we note 
\begin{equation}\label{eq:dimX''}
\dim X_{q''}=\dim C-m+m_0=n-n_0-(m-m_0)
\end{equation}
when $m-m_0\leq n-n_0$. When $m-m_0>n-n_0$, $X_{q''}$ is empty for any such $q''$, so $D(C)$ would be contained in the proper Zariski-closed subset $Q'\times D''$ of $Q$. Thus, we may assume $m-m_0\leq n-n_0$. Fixing $q''\in Q''\setminus D''$, the matrix \eqref{eq:rmatrix} will then have rank $m-m_0$ for all $[c]^{\textbf{\textit{n}}}\in X_{q''}$. Then, considering the map $\Delta$ taking an $m_0\times n_0$ matrix in $\Bbb C^{m_0n_0}$ to the $\binom{n_0}{m_0}$-tuple of determinants of its $m_0\times m_0$ minors and denoting by
\[\mathscr D'(X_{q''}):=\{(q',[c]^{\textbf{\textit{n}}})\in Q'\times X_{q''}\:|\:\Delta(\rho_{c}(q'))=0\}\]
the subset of $Q'\times X_{q''}$ cut out by the homogeneous polynomials 
\[(q',c)\mapsto\Delta(\rho_{c}(q')) \quad\text{on}\quad Q'\times\Bbb C^{n_1+1}\times\cdots\times\Bbb C^{n_\nu+1},\]
we will have $(q',[c]^{\textbf{\textit{n}}})\in\mathscr D'(X_{q''})$ exactly when $[c]^{\textbf{\textit{n}}}\in C$ is a singular point of the subset of $\Bbb P^{\textbf{\textit{n}}}$ cut out by $(q',q'')\in Q$. As $\mathscr D'(X_{q''})$ is the intersection of its Zariski-closure in $Q'\times\overline C$ with the Zariski-open subset $Q'\times C$ of $Q'\times\overline C$, it is Zariski-locally-closed and thus Zariski-constructible in $Q'\times\overline C$. Then, considering the projection $\mathscr Q':Q'\times\overline C\to Q'$ and setting $D'(X_{q''}):=\mathscr Q'(\mathscr D'(X_{q''}))$, $D'(X_{q''})$ is the image of a Zariski-constructible subset under a morphism of finite type and is thus Zariski-constructible in $Q'$ by Chevalley's theorem \cite[Exercise II.3.19]{Hartshorne77}. Therefore, the dimensions of $D'(X_{q''})$ and $\mathscr D'(X_{q''})$ are well-defined and satisfy
\begin{equation}\label{eq:dimDD}
\dim D'(X_{q''})\leq\dim\mathscr D'(X_{q''}).
\end{equation}
Then, considering the projection
$\mathscr P':Q'\times\overline C\to \overline C$
and noting that our above presentation of $\mathscr D'(X_{q''})$ shows it is Zariski-closed in the variety $Q'\times X_{q''}$, denoting by $V_1,\dots,V_v$ the finitely many irreducible components of $\mathscr D'(X_{q''})$, we can see that the restriction
$\mathscr P'|_{V_i}:V_i\to\overline{(\mathscr P'(V_i))}$
is a dominant morphism of finite-type varieties over $\Bbb C$ (in the language of the book of Hartshorne \cite[Section I.3]{Hartshorne77}) for each $i\in\{1,\dots,v\}$, where $\overline{(\mathscr P'(V_i))}$ is the Zariski-closure of the image $\mathscr P'(V_i)$ of $V_i$ under $\mathscr P'$. Therefore, we may apply the fiber-dimension theorem \cite[Exercise II.3.22(b)]{Hartshorne77} to these morphisms. Combining \eqref{eq:dimDD} with the resulting inequalities and \eqref{eq:dimX''}, we then have
\begin{equation}\label{eq:fiberbound}
\dim D'(X_{q''})\leq n-n_0-(m-m_0)+\sup_{[c]^{\textbf{\textit{n}}}\in X_{q''}}\dim\mathscr P'|_{\mathscr D'(X_{q''})}^{-1}([c]^{\textbf{\textit{n}}}).
\end{equation}
Fixing $[c]^{\textbf{\textit{n}}}\in X_{q''}$, we can observe that
\[\mathscr P'|_{\mathscr D'(X_{q''})}^{-1}([c]^{\textbf{\textit{n}}})=\rho_{c}^{-1}(\Delta^{-1}(0))\times\{[c]^{\textbf{\textit{n}}}\},\]
where $\rho_{c}:Q'\to\Bbb C^{m_0n_0}$ is the linear map taking $q'\in Q'$ to the $m_0\times n_0$ matrix \eqref{eq:lmatrix}. We can see that $\rho_c$ is surjective: as $[c]^{\textbf{\textit{n}}}\in C$, we may select $\beta_\alpha\in\{0,\dots,n_\alpha\}$ satisfying $c_{\alpha,\beta_\alpha}\neq0$ for each $\alpha\in\{1,\dots,\nu\}$ such that, for at least one such $\alpha$, $\beta_\alpha\neq n_\alpha$. Then, setting
\[q'_{c,i,j}:=(q'_{c,i,j,1},\dots,q'_{c,i,j,m_0}),\quad
q'_{c,i,j,k}(z):=\delta_{i,k}z_{0,0}^{d_{k,0}-1}z_{0,j}z_{1,\beta_1}^{d_{k,1}}\cdots z_{\nu,\beta_\nu}^{d_{k,\nu}},\]
for $i\in\{1,\dots,m_0\}$ and $j\in\{1,\dots,n_0\}$, we can see that $q'_{c,i,j}\in Q'$ and that $\frac{\partial q'_{c,i,j,k}}{\partial z_{0,l}}(c)$ is nonzero exactly when $k=i$ and $l=j$. Therefore, with $q':=q'_{c,i,j}$, the $(k,l)$th entry of the matrix \eqref{eq:lmatrix} will be nonzero exactly when $(k,l)=(i,j)$, so surjectivity is assured. This shows that
\[\dim Q'-\dim \Bbb C^{m_0n_0}=\dim\textrm{Ker}\rho_{c}=\dim\mathscr P'|_{\mathscr D'(X_{q''})}^{-1}([c]^{\textbf{\textit{n}}})-\dim\Delta^{-1}(0),\]
so directly performing the linear-algebraic computation that
\[\dim \Bbb C^{m_0n_0}-\dim\Delta^{-1}(0)=n_0-m_0+1,\]
we have
\[\dim\mathscr P'|_{\mathscr D'(X_{q''})}^{-1}([c]^{\textbf{\textit{n}}})=\dim Q'-n_0+m_0-1.\]
Combining this equality with \eqref{eq:fiberbound} and our assumption that the right-hand condition of \eqref{eq:parities} holds, we then have that
\[\dim D'(X_{q''})\leq \dim Q'+n-2n_0-(m-2m_0)-1<\dim Q',\]
so $D'(X_{q''})$ has empty Zariski-interior in $Q'$.
Now, consider a nonempty Zariski-open subset $U$ of $Q$. As $D''$ is a proper Zariski-closed subset of $Q''$, $Q'\times D''$ is a proper Zariski-closed subset of $Q'\times Q''$, so there certainly exists some $(q',q'')\in U$ not in $Q'\times D''$. Fixing such $q''$, the set $U\cap(Q'\times\{q''\})$ will then be a nonempty Zariski-open subset of $Q'\times\{q''\}$ and thus cannot be contained in $D'(X_{q''})\times\{q''\}$, which in turn contains $D(C)\cap(Q'\times\{q''\})$. Therefore, we have
\[U\cap(Q'\times\{q''\})\not\subseteq D(C)\cap(Q'\times\{q''\})\]
and thus $U\not\subseteq D(C)$, so $D(C)$ has empty Zariski-interior in $Q$. 
As noted above, this completes the proof that $D=D(A)\cup D(B)\cup D(C)$ is a proper subset of $Q$, which in turn completes the proof altogether.
\end{proof}

\subsection{Involutions of complete intersections in products of complex projective spaces}
\label{sub:involmulti}

Lemma \ref{lem:multidiffeo} shows that any smooth complete intersection cut out in a product of complex projective spaces by polynomials exhibiting the symmetries of interest in Subsection \ref{sub:symmulti} admits involutions of the form of interest in Proposition \ref{pro:multi}.

\begin{lemma}\label{lem:multidiffeo}
For any tuple $\textbf{\textit{n}}:=(n_0,\dots,n_\nu)$ of positive integers and any smooth complete intersection $X$ cut out in the product $\Bbb P^\textbf{\textit{n}}$ of complex projective spaces by polynomials with only real coefficients whose constituent monomials each have even power of $z_{0,0}$ and positive power of some coordinate function other than any of $z_{0,n_0},\dots,z_{\nu,n_\nu}$, the maps
\begin{equation*}
\begin{gathered}
([z_0],\dots,[z_\nu])\mapsto ([-z_{0,0}:z_{0,1}:\cdots:z_{0,n_0}],[z_1],\dots,[z_\nu]),\\
([z_0],\dots,[z_\nu])\mapsto([\overline z_0],\dots,[\overline z_\nu])
\end{gathered}
\end{equation*}
on $\Bbb P^\textbf{\textit{n}}$ respectively restrict to commuting smooth involutions $a$ and $c$ of $X$ which act as the antipodal map in one complex coordinate and complex conjugation in suitable complex local coordinates about their shared fixed point
\[([0:\cdots:0:1],\dots,[0:\cdots:0:1])\in X.\]
\end{lemma}

The proof of Lemma \ref{lem:multidiffeo} is identical to that of Lemma \ref{lem:cpletediffeo}, so we avoid reproducing it here. With $X$ (of complex dimension at least 2) as in Lemma \ref{lem:multidiffeo}, we may also observe as in Subsection \ref{sub:diffeos} that $X/a$ will be singular when all polynomials cutting out $X$ have odd degree in the zeroth factor, while otherwise, a generic choice of polynomials as in Lemma \ref{lem:multipolys} will cut out a smooth complete intersection in $\Bbb P^\textbf{\textit{n}}$ which can be given the structure of a holomorphic double branched cover with compatible real structure whose branch divisor contains a real point.

As was mentioned in Subsection \ref{sub:multi-defex}, taking $p_1$ and $p_2$ as in \eqref{eq:p1p2}, the polynomials cutting out the complex hypersurfaces $X_{2,m,n}$ are as in Lemma \ref{lem:multidiffeo}, providing explicit descriptions of how smooth involutions of the smooth manifolds $X(m,n)$ as in the lemma arise. Recalling the hyperelliptic involution $\iota$ of the surface $\Sigma_g$ of genus $g$ alongside the construction of the smooth 4-manifold $X(m,n)$ described in Subsection \ref{sub:multi-defex} as the quotient by $\iota\times\iota$ of the blow-up of $\Sigma_{m-1}\times\Sigma_{n-1}$ at the fixed points of $\iota\times\iota$, we may also construct such involutions from this perspective. To this end, we consider an orientation-reversing involution $r$ on $\Sigma_g$ which commutes with $\iota$ such that, about each fixed point of $\iota$, there exist complex coordinates on $\Sigma_g$ such that $\iota$ and $r$ act as the antipodal map and complex conjugation respectively (as will be true taking $\iota$ and $r$ as in Figures \ref{fig:Hyperelliptic} and \ref{fig:Reflection}). It is then straightforward to extend the involutions $\iota\times\textrm{id}$ and $r\times r$ to the blow-up of $\Sigma_{m-1}\times\Sigma_{n-1}$ at the fixed points of $\iota\times\iota$. Since the resulting involutions commute with $\iota\times\iota$, they descend to commuting smooth involutions $a$ and $c$ of $X(m,n)$. These involutions then satisfy the properties of interest in Proposition \ref{pro:mainprop}, where we take $x$ to be the image in $X(m,n)$ of any fixed point of $\iota\times r$.

\begin{figure}
\begin{center}
\includegraphics[width=.52\textwidth]
{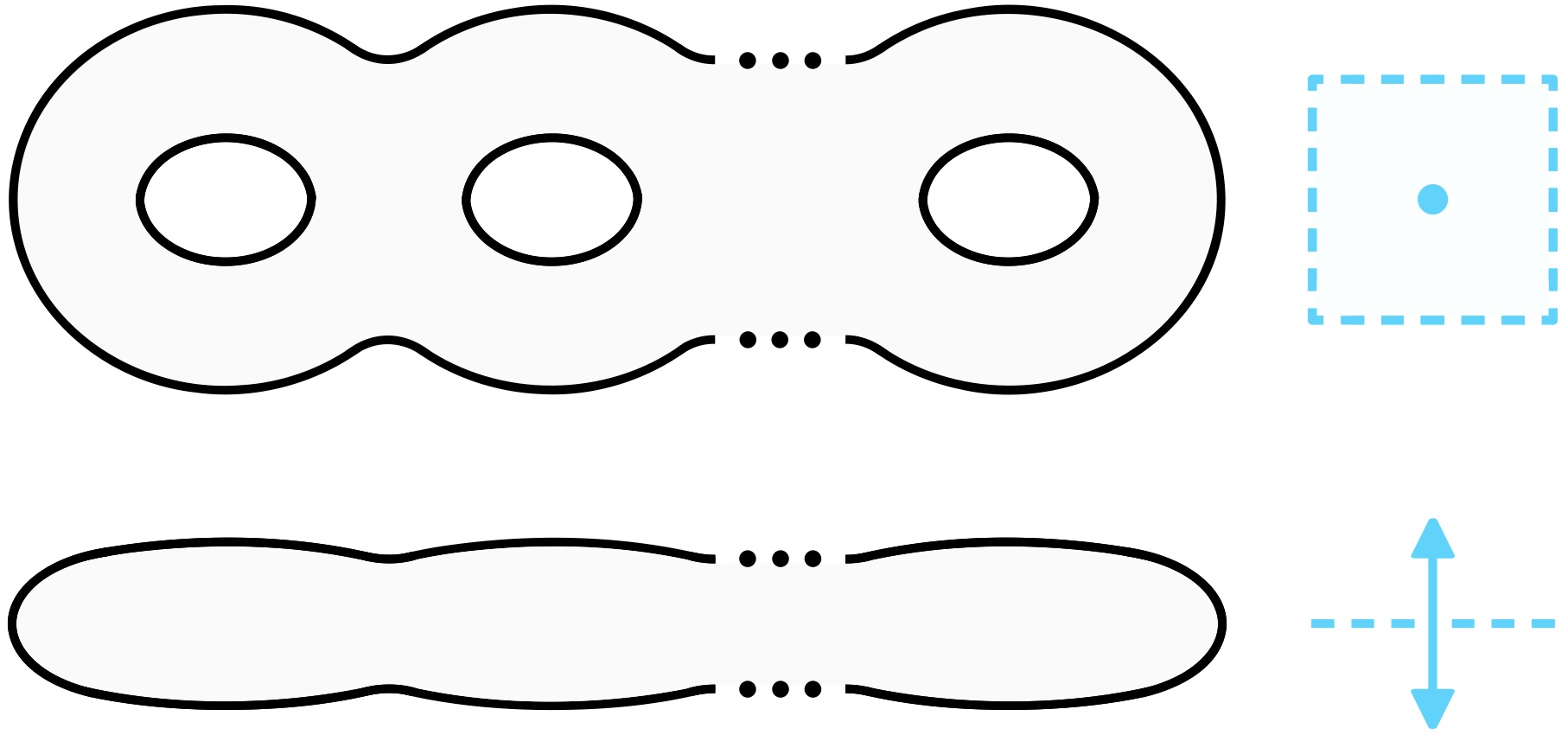}
\caption{\label{fig:Reflection}An orientation-reversing involution $r$ of the surface $\Sigma_g$ of genus $g$, suitably compatible with the hyperelliptic involution $\iota$ visualized in Figure \ref{fig:Hyperelliptic} to allow for the construction of involutions $a$ and $c$ as in Lemma \ref{lem:multidiffeo} on $X=X(m,n)$ from the perspective described in Subsection \ref{sub:involmulti}}
\end{center}
\end{figure}

In addition to exhibiting involutions as in Lemma \ref{lem:multidiffeo} on the smooth complete intersections $X_{2,m,n}$ as in Subsection \ref{sub:multi-defex}, Lemma \ref{lem:multidiffeo} can be combined with Lemma \ref{lem:multipolys} to directly give Proposition \ref{pro:multi}.

\section{Boundary Dehn twists as commutators}
\label{sec:proproof}

In this section, we prove Proposition \ref{pro:mainprop}, which (when $f$ is the identity diffeomorphism) details a correspondence between diffeomorphisms $a^\circ,c^\circ$ of $X^\circ$ rel boundary whose commutator represents the boundary Dehn twist on $X^\circ$ and diffeomorphisms $a,c$ of $X$ which admit certain local descriptions near a common fixed point and whose commutator is smoothly isotopic to the identity through diffeomorphisms fixing a neighborhood of this point. This result (alongside Remark \ref{rmk:genprop} and Lemma \ref{lem:connsum}) is combined with Propositions \ref{pro:cpleteint} and \ref{pro:multi} and the discussion of further spaces $X$ admitting such diffeomorphisms $a,c$ in Section \ref{sec:intro} to give the result of Theorem \ref{thm:main}.

\begin{proposition}\label{pro:mainprop}
Consider a smooth oriented connected closed manifold $X$ of dimension $n\geq3$, local coordinates $\psi:U\to\Bbb R^n$ on $X$, and any $f\in\operatorname{Diff}^+(X)$ which restricts to the identity on $U$. Writing $X^\circ:=X\setminus\psi^{-1}(B^n)$, from either of the following, we may construct the other:
\begin{itemize}
	\item Diffeomorphisms $a^\circ,c^\circ\in\operatorname{Diff}^{+}(X^\circ,\partial)$ such that $[a^\circ,c^\circ]f|_{X^\circ}$ represents the boundary Dehn twist on $X^\circ$
	\item Diffeomorphisms $a,c\in\operatorname{Diff}^{+}(X)$ such that $([a,c]f)|_{X^\circ}$ is smoothly isotopic to the identity on $X^\circ$ rel boundary and that, for $r_k:=\emph{diag}(-1,-1,1)$ and $r_i:=\emph{diag}(1,-1,-1)$, we have
	\begin{equation}\label{eq:acr}
		\psi\circ a\circ\psi^{-1}=r_k\oplus I_{n-3},\quad\psi\circ c\circ \psi^{-1}=r_i\oplus I_{n-3}.
	\end{equation}
\end{itemize}
\end{proposition}

When $n\geq5$, we apply Remark \ref{rmk:genprop} to put the involutions constructed in the previous sections into the exact form of interest in Proposition \ref{pro:mainprop}.

\begin{remark}\label{rmk:genprop}
For $X$, $\psi$, and $f$ as in Proposition \ref{pro:mainprop}, consider diffeomorphisms $a,c'\in\operatorname{Diff}^+(X)$ such that $([a,c']f)|_{X^\circ}$ is smoothly isotopic to the identity on $X^\circ$ rel boundary, $a$ is as in \eqref{eq:acr}, and we have 
\[\psi\circ c'\circ\psi^{-1}=r_i\oplus(-I_{2m})\oplus I_{n-3-2m}\] 
for some $m\leq(n-3)/2$. A smooth path in $\text{SO}(2m)$ taking $-I_{2m}$ to $I_{2m}$ directly gives rise to a smooth isotopy from $c'$ to some $c\in\operatorname{Diff}^+(X)$ such that $a,c$ are exactly as in Proposition \ref{pro:mainprop}.
\end{remark}

The boundary Dehn twist on $X^\circ$ has trivial image in the abelianization of the smooth mapping class group rel boundary exactly when it can be realized as a product of some $g\in\Bbb N_+$ commutators in the smooth mapping class group rel boundary. Notably, Proposition \ref{pro:mainprop} shows that this will be the case if and only if there exist $a,c\in\operatorname{Diff}^+(X)$ as in the proposition for $f$ a composition $[a_2,c_2]\cdots[a_g,c_g]$ of $(g-1)$ commutators of diffeomorphisms $a_i,c_i\in\operatorname{Diff}^+(X)$ which each restrict to the identity on $U$. 
In fact, generalizing Remark \ref{rmk:genprop}, it can be shown using elementary linear algebra and smooth isotopies that we may produce diffeomorphisms $a,c\in\operatorname{Diff}^+(X)$ and $a^\circ,c^\circ\in\operatorname{Diff}^+(X^\circ,\partial)$ as in Proposition \ref{pro:mainprop} for some such $f$ from the data of any $g$ pairs of diffeomorphisms $a_i',c_i'\in\operatorname{Diff}^+(X)$ with a common fixed point $x\in X$ such that the composition $[a_1',c_1']\cdots[a_g',c_g']$ of their commutators is smoothly isotopic to the identity through diffeomorphisms whose differentials each act as the identity at $x$ and such that the sum $\nu_1+\cdots+\nu_g$ is odd, where $\nu_i\in\{0,1\}$ is the parity of the dimension of the maximal subspace of $T_xX$ on which all generalized eigenvalues of the differentials $d(a_i')_x$ and $d(c_i')_x$ are negative real numbers. We omit this argument for brevity, but note it is straightforward.

In Subsection \ref{sub:spin}, we introduce paths in $\text{SO}(n)$ whose commutator is a loop with homotopy class the nontrivial element of $\pi_1(SO(n))$. We then use these paths to formalize the constructions of interest in Proposition \ref{pro:mainprop} and prove the proposition in Subsection \ref{sub:commutators}. In Subsection \ref{sub:connsum}, we apply the proposition to observe that concrete presentations for the boundary Dehn twists on the punctures of two manifolds of the same dimension at least 4 given by products of $g\in\Bbb N_+$ commutators in the smooth mapping class group rel boundary can be combined to give rise to such a presentation for the boundary Dehn twist on the puncture of their connected sum. We then discuss associations between these results and the study of smooth orientable $X$-bundles over orientable closed surfaces which admit sections whose normal bundles are not spin in Subsection \ref{sub:mappingtori}, and conclude by applying these associations and arguments \cite[Subsection 4.2]{BaragliaKonno25} of Baraglia and Konno in Subsection \ref{sub:Torelli} to verify Corollary \ref{cor:ourBKlem}, which states that the boundary Dehn twists shown by Baraglia and Konno to be nontrivial in the smooth mapping class group rel boundary \cite[Theorem 1.4]{BaragliaKonno25} even have nontrivial image in the abelianization of the Torelli subgroup of the smooth mapping class group rel boundary, despite Theorem \ref{thm:main} showing that all such boundary Dehn twists on punctured complete intersections will have trivial image in the abelianization of the full smooth mapping class group rel boundary.

\subsection{A commutator loop in the special orthogonal group}
\label{sub:spin}

Consider
\begin{equation}\label{eq:Rik}
R_{k}(t):=\begin{bmatrix}
\cos (\pi t) & \sin (\pi t) & \\ -\sin (\pi t) & \cos (\pi t) & \\ & & 1
\end{bmatrix},\:R_{i}(t):=\begin{bmatrix}
1 & & \\ & \cos (\pi t) & \sin (\pi t) \\ & -\sin (\pi t) & \cos (\pi t)
\end{bmatrix}
\end{equation}
for $t\in[0,1]$, so for $r_k,r_i$ as in Proposition \ref{pro:mainprop}, we have
\[R_k(0)=R_i(0)=I_3,\quad R_k(1)=r_k,\quad R_i(1)=r_i.\]

\begin{lemma}\label{lem:spin}
For $n\geq3$, the commutator loop $[R_k\oplus I_{n-3},R_i\oplus I_{n-3}]$ represents the nontrivial element of $\pi_1(\text{SO}(n))$.
\end{lemma}

\begin{proof}
Denoting by
\[\Bbb H:=\Bbb R\langle i,j\rangle/(i^2=j^2=-1,\,ij=-ji)\]
the quaternions, by $\text{Spin}(3)$ the group of unit quaternions, and by 
\[\text{Im}(\Bbb H):=\Bbb Ri\oplus\Bbb Rj\oplus\Bbb Rk\quad(k:=ij)\]
the vector space of purely imaginary quaternions, the map
\[\text{Spin}(3)\to\text{SO}(3)\cong\text{SO}(\text{Im}(\Bbb H)),\quad q\mapsto (v\mapsto q^{-1}vq)\]
is the universal cover of $\text{SO}(3)$. The paths in $\text{Spin}(3)$ defined for $t\in[0,1]$ by $Q_k(t):=e^{\pi kt/2}$ and $Q_i(t):=e^{\pi it/2}$ are respectively lifts of the paths $R_k,R_i$ in $\text{SO}(3)$, so the commutator path $[Q_k,Q_i]$ is a lift of $[R_k,R_i]$. Now, we have
\[[Q_k(1),Q_i(1)]=(-k)(-i)ki=-1\neq1=[Q_k(0),Q_i(0)],\]
so the commutator loop $[R_k,R_i]$ must represent the nontrivial element of $\pi_1(SO(3))\cong\Bbb Z/2\Bbb Z$. For $n\geq3$, noting that the inclusion $\Bbb R^3\hookrightarrow\Bbb R^n$ into the first three factors induces an isomorphism
$\pi_1(SO(3))\cong\pi_1(SO(n))$,
we then get exactly the desired result.
\end{proof}

\subsection{Boundary Dehn twists as commutators}
\label{sub:commutators}

Consider the setting of Proposition \ref{pro:mainprop}. Fixing $a,c\in\operatorname{Diff}^{+}(X)$ as in the proposition, we now construct $a^\circ,c^\circ\in\operatorname{Diff}^{+}(X^\circ,\partial)$ as desired, then present the reverse construction to verify the proposition. To this end, because $[a,c]f$ restricts to the identity on $U$, setting $\nu(\partial X^\circ):=\psi^{-1}(B^n_2\setminus B^n)$ (for $B^n_r$ the open ball of radius $r$ about the origin in $\Bbb R^n$), $([a,c]f)|_{X^\circ\setminus\nu(\partial X^\circ)}$ will still be smoothly isotopic to the identity on $X^\circ\setminus\nu(\partial X^\circ)$ rel boundary. Smoothing the paths $R_k\oplus I_{n-3}$ and $R_i\oplus I_{n-3}$ as in \eqref{eq:Rik} near 0 and 1, we produce smooth paths
\begin{equation}\label{eq:rhoki}
\rho_k,\rho_i:[0,1]\to\text{SO}(n)
\end{equation}
which both evaluate to the identity near 0 and respectively evaluate to $r_k\oplus I_{n-3}$ and $r_i\oplus I_{n-3}$ near $1$. We then define $a^\circ,c^\circ\in\operatorname{Diff}^{+}(X^\circ,\partial)$ by
\begin{gather*}
a^\circ(x):=(\psi^{-1}\circ\rho_k(|\psi(x)|-1)\circ\psi)(x),\\
 c^\circ(x):=(\psi^{-1}\circ\rho_i(|\psi(x)|-1)\circ\psi)(x)
\end{gather*}
when $x\in\nu(\partial X^\circ)$ and
\[a^\circ(x):=a(x),\quad c^\circ(x):=c(x)\]
otherwise.

\begin{lemma}\label{lem:biglemma}
For diffeomorphisms $a,c\in\operatorname{Diff}^{+}(X)$ as in Proposition \ref{pro:mainprop} and $a^\circ,c^\circ\in\operatorname{Diff}^{+}(X^\circ,\partial)$ constructed from $a,c$ as above, $[a^\circ,c^\circ]f|_{X^\circ}$ represents the boundary Dehn twist on $X^\circ$.
\end{lemma}

\begin{proof}
As $([a^\circ,c^\circ]f)|_{X^\circ\setminus\nu(\partial X^\circ)}=([a,c]f)|_{X^\circ\setminus\nu(\partial X^\circ)}$ is smoothly isotopic to the identity on $X^\circ\setminus\nu(\partial X^\circ)$ rel boundary, $[a^\circ,c^\circ]f|_{X^\circ}$ will be smoothly isotopic (rel boundary) to the diffeomorphism of $X^\circ$ which equals $[a^\circ,c^\circ]$ on $\nu(\partial X^\circ)$ and the identity elsewhere. Considering the map
\[\phi:\nu(\partial X^\circ)\to[0,1)\times S^{n-1},\quad x\mapsto(|\psi(x)|-1,\psi(x)/|\psi(x)|),\]
we can observe that
\[\phi\circ[a^\circ,c^\circ]\circ\phi^{-1}:(t,\omega)\mapsto(t,[\rho_k(t),\rho_i(t)]\omega).\]
Lemma \ref{lem:spin} then shows that the homotopy class of the loop $[\rho_k,\rho_i]$ generates $\pi_1(SO(n))$, so we can see from the definition of the boundary Dehn twist presented in Section \ref{sec:intro} that the desired result holds.
\end{proof}

For $a^\circ,c^\circ\in\operatorname{Diff}^{+}(X^\circ,\partial)$ as in Proposition \ref{pro:mainprop}, we now construct $a,c\in\operatorname{Diff}^+(X)$ as desired to complete the proof of the proposition. We first consider local coordinates $\psi':U'\to\Bbb R^n$ on $X$ such that $\psi'\circ\psi^{-1}$ is the standard contraction $\Bbb R^n\to B^n$ alongside $f'\in\operatorname{Diff}^+(X)$ which restricts to the identity on $U'$ and  satisfies that $f'|_{X^\circ}$ is smoothly isotopic to $f|_{X^\circ}$ rel boundary. After smoothly isotoping $a^\circ$ and $c^\circ$, we may also arrange that these diffeomorphisms restrict to the identity on $(\psi')^{-1}(B^n_2)\cap X^\circ$ and that $([a^\circ,c^\circ]f')|_{X\setminus(\psi')^{-1}(B^n_2)}$ represents the boundary Dehn twist on $X\setminus(\psi')^{-1}(B^n_2)$.
For $\rho_k,\rho_i$ as in \eqref{eq:rhoki}, we then define $a,c\in\operatorname{Diff}^+(X)$ by 
\begin{gather*}
a(x):=((\psi')^{-1}\circ\rho_k(2-\max(1,|\psi'(x)|))\circ\psi')(x),\\
c(x):=((\psi')^{-1}\circ\rho_i(2-\max(1,|\psi'(x)|))\circ\psi')(x)
\end{gather*}
when $x\in(\psi')^{-1}(B^n_2)$ and
\[a(x):=a^\circ(x),\quad c(x):=c^\circ(x)\]
otherwise.
In light of Lemma \ref{lem:biglemma}, we then need only show Lemma \ref{lem:biglemmabackwards} to complete the proof of Proposition \ref{pro:mainprop}.

\begin{lemma}\label{lem:biglemmabackwards}
For diffeomorphisms $a^\circ,c^\circ\in\operatorname{Diff}^{+}(X^\circ,\partial)$ as in Proposition \ref{pro:mainprop} and $a,c\in\operatorname{Diff}^+(X)$ constructed from $a^\circ,c^\circ$ as above, $([a,c]f)|_{X^\circ}$ is smoothly isotopic to the identity on $X^\circ$ rel boundary and \eqref{eq:acr} is satisfied.
\end{lemma}

\begin{proof}
As \eqref{eq:acr} holds by construction, since $f'|_{X^\circ}$ is smoothly isotopic to $f|_{X^\circ}$ rel boundary, we need only show that $([a,c]f')|_{X^\circ}$ is smoothly isotopic to the identity on $X^\circ$ rel boundary. To this end, consider the map
\[\phi:(\psi')^{-1}(B^n_3\setminus B^n)\to[0,2)\times S^{n-1},\quad x\mapsto(|\psi'(x)|-1,\psi'(x)/|\psi'(x)|).\]
Because we arranged that $([a^\circ,c^\circ]f')|_{X\setminus(\psi')^{-1}(B_2^n)}$ represents the boundary Dehn twist on $X\setminus(\psi')^{-1}(B_2^n)$, we have from Lemma \ref{lem:spin} and the definition of the boundary Dehn twist in Section \ref{sec:intro} that $([a,c]f')|_{X^\circ}$ is smoothly isotopic (rel boundary) to the diffeomorphism $\varphi$ of $X^\circ$ which satisfies
\begin{gather*}
\phi\circ\varphi\circ\phi^{-1}|_{[0,1)\times S^{n-1}}:(t,\omega)\mapsto(t,[\rho_k(1-t),\rho_i(1-t)]\omega), \\
\phi\circ\varphi\circ\phi^{-1}|_{[1,2)\times S^{n-1}}:(t,\omega)\mapsto(t,[\rho_k(t-1),\rho_i(t-1)]\omega)
\end{gather*}
and acts as the identity otherwise. In turn, we can directly observe that this diffeomorphism is smoothly isotopic to the identity on $X^\circ$ rel boundary, completing the proof.
\end{proof}

\subsection{Connected sums}
\label{sub:connsum}

We now apply Proposition \ref{pro:mainprop} to prove Lemma \ref{lem:connsum}, which makes the observation that if the boundary Dehn twists on the punctures of some manifolds $X_1$ and $X_2$ of the same dimension $n\geq4$ can each be presented as a product of $g\in\Bbb N_+$ commutators in the smooth mapping class group rel boundary, the same will be true of the boundary Dehn twist on the puncture of the connected sum $X:=X_1\,\#\,X_2$ (so notably, whenever the boundary Dehn twists on $X_1^\circ$ and $X_2^\circ$ both have trivial image in the abelianization of the smooth mapping class group rel boundary, the same is true of the boundary Dehn twist on $X^\circ$).

\begin{lemma}\label{lem:connsum}
Fixing $n\geq4$ and $g\in\Bbb N_+$, for each $j\in\{1,2\}$, consider a smooth oriented connected closed $n$-manifold $X_j$ alongside diffeomorphisms $a_{ij}^\circ,c_{ij}^\circ\in\operatorname{Diff}^+(X_j^\circ,\partial)$ such that $[a_{1j}^\circ,c_{1j}^\circ]\cdots[a_{gj}^\circ,c_{gj}^\circ]$ represents the boundary Dehn twist on $X_j^\circ$. Setting $X:=X_1\,\#\,X_2$, we may produce diffeomorphisms $a_i^\circ,c_i^\circ\in\operatorname{Diff}^+(X^\circ,\partial)$ such that $[a_1^\circ,c_1^\circ]\cdots[a_g^\circ,c_g^\circ]$ represents the boundary Dehn twist on $X^\circ$.
\end{lemma}

\begin{proof}
Fixing notation as in the lemma statement and $j\in\{1,2\}$, we consider local coordinates $\psi_j:U_j\to\Bbb R^n$ on $X_j$ and say $X_j^\circ=X_j\setminus\psi_j^{-1}(B^n)$. Writing $f_j:=[a_{2j},c_{2j}]\cdots[a_{gj},c_{gj}]$ for $a_{ij}$ and $c_{ij}$ ($i\in\{2,\dots,g\}$) the diffeomorphisms of $X_j$ which both restrict to the identity on $\psi_j^{-1}(B^n)$ and respectively restrict to $a_{ij}^\circ$ and $c_{ij}^\circ$ (which we smoothly isotope so that each restricts to the identity on $U_j\cap X_j^\circ$) on $X_j^\circ$, as in Proposition \ref{pro:mainprop}, we may construct $a_{1j},c_{1j}\in\operatorname{Diff}^+(X_j)$ such that $([a_{1j},c_{1j}]f_j)|_{X_j^\circ}$ is smoothly isotopic to the identity on $X_j^\circ$ rel boundary and that \eqref{eq:acr} is satisfied for $\psi:=\psi_j$, $a:=a_{1j}$, and $c:=c_{1j}$. Denoting by $B_{1/4}^n(v)$ the ball of radius $1/4$ about $v\in\Bbb R^n$ (and writing $B_{1/4}^n:=B_{1/4}^n(0)$), we consider the connected sum 
\[X:=(X_1\setminus\psi_1^{-1}(B_{1/4}^n))\cup_{S^{n-1}}X_2^\circ\]
alongside the resulting diffeomorphisms $a_{i}:=a_{i1}\,\#\,a_{i2}$ and $c_i:=c_{i1}\,\#\,c_{i2}$ for each $i\in\{1,\dots,g\}$. Restricting the local coordinates $\psi_1:U_1\to\Bbb R^n$ on $X_1$ to the set $U:=\psi_1^{-1}(B_{1/4}^n(0,0,0,1/2,0,\dots,0))$ and reparametrizing, we may produce local coordinates $\psi:U\to\Bbb R^n$ on $X$ such that \eqref{eq:acr} is satisfied for $a:=a_1$ and $c:=c_1$. Writing $f:=[a_2,c_2]\cdots[a_g,c_g]$ and $X^\circ:=X\setminus\psi^{-1}(B^n)$, we may then observe that $f$ restricts to the identity on $X^\circ\setminus(X_1^\circ\cup X_2^\circ)$ (and, we note, on $U$) and that $([a_{1},c_{1}]f)|_{X_j^\circ}=([a_{1j},c_{1j}]f_j)|_{X_j^\circ}$ is smoothly isotopic to the identity on $X_j^\circ$ rel boundary for each $j\in\{1,2\}$, so $([a_1,c_1]f)|_{X^\circ}$ is smoothly isotopic to the identity on $X^\circ$ rel boundary. Applying Proposition \ref{pro:mainprop} then completes the proof.
\end{proof}


\subsection{Non-spin bundles}
\label{sub:mappingtori}

For a smooth oriented manifold $X$, we consider $g\in\Bbb N_+$ pairs of diffeomorphisms $a_i,c_i\in\operatorname{Diff}^+(X)$ and a smooth isotopy $H$ from the composition $[a_1,c_1]\cdots[a_g,c_g]$ of commutators of these pairs to the identity. Considering the standard CW decomposition 
\[\Sigma_g=e^0\cup e^1_{a_1}\cup e^1_{c_1}\cup\cdots\cup e^1_{a_g}\cup e^1_{c_g}\cup e^2_H\]
of the surface of genus $g$ with one 0-cell $e^0$, $2g$ 1-cells $e^1_{a_i},e^1_{c_i}$, and one 2-cell $e^2_H$, we first consider the mapping tori $X_{a_i}\to e^0\cup e^1_{a_i}$ and $X_{c_i}\to e^0\cup e^1_{c_i}$ for each $i$, the wedge of which gives an $X$-family over the 1-skeleton of $\Sigma_g$. We then use $H$ to direct how to extend this family from the 1-skeleton to the 2-cell, giving rise to a smooth orientable $X$-bundle $E_{F,H}\to\Sigma_g$ with monodromy diffeomorphisms 
$F:=(a_1,c_1,\dots,a_g,c_g)$. 
A common fixed point $x\in X$ of these diffeomorphisms and each diffeomorphism of the isotopy $H$ gives rise to a section $s_{F,H,x}:\Sigma_g\to E_{F,H}$ of the bundle whose normal bundle we denote by $\nu(s_{F,H,x})$.

\begin{lemma}\label{lem:mappingtorus}
Consider $a_1:=a,c_1:=c\in\operatorname{Diff}^+(X)$ as in Proposition \ref{pro:mainprop} for $f$ a composition $[a_2,c_2]\cdots[a_g,c_g]$ of $(g-1)\in\Bbb N$ commutators of diffeomorphisms $a_i,c_i\in\operatorname{Diff}^+(X)$ which each restrict to the identity on $U$. Denoting by $H$ a smooth isotopy from $[a_1,c_1]\cdots[a_g,c_g]=[a,c]f$ to the identity through diffeomorphisms each fixing $\psi^{-1}(B^n)$ and setting $x:=\psi^{-1}(0)$, we have that $\nu(s_{F,H,x})$ is not spin.
\end{lemma}

\begin{proof}
We may observe that the bundle $\nu(s_{F,H,x})\to\Sigma_g$ has monodromy maps 
\[dF_x=(d(a_1)_x,d(c_1)_x,\dots,d(a_g)_x,d(c_g)_x).\]
With $r_k,r_i$ as in Proposition \ref{pro:mainprop}, we then have that 
\[d(a_1)_x=r_k\oplus I_{n-3},\quad d(c_1)_x=r_i\oplus I_{n-3},\quad d(a_i)_x=d(c_i)_x=I_n\]
for $i\in\{2,\dots,g\}$. Using the paths $\rho_k,\rho_i$ as in \eqref{eq:rhoki} to trivialize the bundle $\nu(s_{F,H,x})\to\Sigma_g$ over the 1-skeleton of $\Sigma_g$, we can observe that the bundle will be spin exactly when the homotopy class of the loop $[\rho_k,\rho_i]$ in $\pi_1(\text{SO}(n))$ is trivial, which Lemma \ref{lem:spin} shows is not the case.
\end{proof}

In fact, for any smooth oriented connected closed manifold $X$ of dimension at least 3, any $g\in\Bbb N_+$, and any smooth orientable $X$-bundle $E\to\Sigma_g$ which admits a section $s:\Sigma_g\to E$ whose normal bundle is not spin, it can be shown (using arguments resembling  those found in work \cite[Proposition 2.1]{Lin25} of Y. Lin or those mentioned directly after Remark \ref{rmk:genprop}) that there exist a bundle $E_{F,H}\to\Sigma_g$ with monodromy diffeomorphisms $a_1,c_1,\dots,a_g,c_g$ as in Lemma \ref{lem:mappingtorus} and a bundle isomorphism $\varphi:E\to E_{F,H}$ such that $\varphi\circ s$ is the section $s_{F,H,x}$ of interest in the lemma, where $H$ is a smooth isotopy as in the lemma statement.

For a smooth orientable $X$-bundle $E\to\Sigma_g$, if $E$ is spin, the normal bundle of any section $s:\Sigma_g\to E$ must be spin. In fact, assuming that $X$ is spin and simply-connected, the converse is also true. Specifically, the second Stiefel-Whitney classes of $T\Sigma_g$ and $TX$ will then vanish because $\Sigma_g$ and $X$ are spin. Assuming $E$ is not spin, the second Stiefel-Whitney class of $TE$ will then not vanish. By simply-connectedness of $X$, we have that $E\to\Sigma_g$ admits a section $s:\Sigma_g\to E$ and that $H^1(X;\Bbb Z/2\Bbb Z)$ is trivial, and the Serre spectral sequence then gives that the normal bundle of $s$ has nonvanishing second Stiefel-Whitney class, and is thus not spin. This allows us to simplify the above statements when $X$ is spin and simply-connected.

\subsection{The Torelli subgroup}
\label{sub:Torelli}

In this subsection, we review the arguments \cite[Subsections 4.1-4.2]{BaragliaKonno25} used by Baraglia and Konno in their elegant proof that many boundary Dehn twists \cite[Theorem 1.4]{BaragliaKonno25} are nontrivial in the smooth mapping class group rel boundary, then discuss how these arguments can be readily applied to show that these boundary Dehn twists even have nontrivial image in the abelianization of the Torelli subgroup of the smooth mapping class group rel boundary. To this end, first consider the diffeomorphisms $a^\circ,c^\circ$ of $D^4\cong(S^4)^\circ$ obtained by applying Proposition \ref{pro:mainprop} with $a$ and $c$ the maps $r_k\oplus I_2$ and $r_i\oplus I_2$ on $X:=S^4\subset\Bbb R^5$
and $\psi:U\to\Bbb R^4$ appropriate local coordinates on $X$ about the point $(0,0,0,0,1)\in X$. 
Baraglia and Konno considered diffeomorphisms $\sigma_1',\sigma_2'$ of $D^4$ \cite[Subsections 4.1-4.2]{BaragliaKonno25} which, in our notation, are respectively smoothly isotopic to $a^\circ$ and $c^\circ$. They extended these diffeomorphisms to diffeomorphisms $\sigma_1,\sigma_2$ of an arbitrary smooth simply-connected closed 4-manifold $X\supset D^4$ by setting them equal to the identity on $X\setminus D^4$. Note we may then consider the puncture $X^\circ$ of $X$ such that $[\sigma_1,\sigma_2]|_{X^\circ}$ represents the boundary Dehn twist on $X^\circ$ (though we note to avoid confusion that the point here is different from in our setting: $\sigma_1$ and $\sigma_2$ do not then restrict to the identity on a neighborhood of $\partial X^\circ$ and do not give rise to a commutator representative for the boundary Dehn twist on $X^\circ$ in the smooth mapping class group rel boundary). Baraglia and Konno then assumed for the sake of contradiction that the boundary Dehn twist on $X^\circ$ is trivial in the smooth mapping class group rel boundary, giving rise to a smooth $X$-bundle $E\to T^2$ with monodromy diffeomorphisms $\sigma_1$ and $\sigma_2$ which admits a section $s:T^2\to E$ whose normal bundle is not spin. As $\sigma_1$ and $\sigma_2$ act trivially on the second cohomology of $X$, Baraglia and Konno observed that the maximal positive-definite subbundle $H^+(E)\to T^2$ of the second cohomology bundle (with $\Bbb R$-coefficients) of $E\to T^2$ has vanishing second Stiefel-Whitney class. These authors then presented arguments \cite[Subsection 4.2]{BaragliaKonno25} which used the section $s$ alongside certain conditions that $X$ was then assumed to satisfy to contradict this observation, completing their proof that the boundary Dehn twists on the punctures of all spaces $X$ satisfying these assumptions are nontrivial in the smooth mapping class group rel boundary. Lemma \ref{lem:BKlem} results from applying their exact arguments in wider generality; for clarity, we reproduce these arguments after stating the lemma.

\begin{lemma}[Result of arguments \cite{BaragliaKonno25} of Baraglia and Konno]\label{lem:BKlem}
Let $X$ be a spin smooth simply-connected closed 4-manifold with signature $\sigma(X)$ congruent to 16 modulo 32 and positive second Betti number $b^+(X)$ congruent to 3 modulo 4. Say $X$ admits a spin$^c$ structure $\mathfrak s$ whose first Chern class $c_1(\mathfrak s)$ is divisible by 32 and whose Seiberg-Witten invariant $\emph{SW}(X,\mathfrak s)$ is odd. Consider a smooth orientable $X$-bundle $E\to B$ over an orientable closed surface $B$ whose monodromy action fixes the isomorphism class of $\mathfrak s$ and whose total space is not spin. The maximal positive-definite subbundle $H^+(E)\to B$ of the second cohomology bundle of $E\to B$ must have nonvanishing second Stiefel-Whitney class $w_2(H^+(E))$.
\end{lemma}

\begin{proof}
Because the monodromy action of $E\to B$ fixes the isomorphism class of $\mathfrak s$, a result \cite[Proposition 2.1]{Baraglia19} of Baraglia states that there exists a families spin$^c$ structure $\mathfrak s_E$ on the vertical tangent bundle of $E$ whose restriction to each fiber $X_b$ is $\mathfrak s$ (under an association $X_b\cong X$). This fact can be combined with our assumptions that $\sigma(X)\equiv16\pmod{32}$ and $c_1(\mathfrak s)\equiv0\pmod{32}$ to show, using a computation \cite[Lemma 4.1]{BaragliaKonno25} of Baraglia and Konno, that the first Chern class $c_1(\slashed D_E)$ of the families index of the spin$^c$ Dirac operators associated to $(E,\mathfrak s_E)$ is congruent modulo 2 to the second Stiefel-Whitney class of the normal bundle of any section $s:B\to E$. In fact, as in Subsection \ref{sub:mappingtori}, we may apply simply-connectedness of $X$ and the Serre spectral sequence to show that there exists such a section $s$ whose normal bundle will have nonvanishing second Stiefel-Whitney class, so $c_1(\slashed D_E)$ must be odd. $c_1(\slashed D_E)$ was shown \cite[Corollary 1.3]{BaragliaKonno22} by Baraglia and Konno to be congruent modulo 2 to $w_2(H^+(E))$ whenever $X$ is a closed oriented smooth manifold with $b^+(X)\equiv3\pmod4$, $b_1(X)=0$, and $\text{SW}(X,\mathfrak s)$ odd, so since $(X,\mathfrak s)$ satisfies these constraints, the proof is complete.
\end{proof}

Now, say the boundary Dehn twist on the puncture of some smooth simply-connected closed 4-manifold $X$ can be represented by some diffeomorphism of the form $[a_1^\circ,c_1^\circ]\cdots[a_g^\circ,c_g^\circ]$ for $a_i^\circ,c_i^\circ\in\operatorname{Diff}^+(X^\circ,\partial)$. From Proposition \ref{pro:mainprop}, we get diffeomorphisms $a_i,c_i\in\operatorname{Diff}^+(X)$ such that $[a_1,c_1]\cdots[a_g,c_g]$ is smoothly isotopic to the identity on $X$ through diffeomorphisms fixing a neighborhood $\psi^{-1}(B^4)\subset X$ as in the proposition. Subsection \ref{sub:mappingtori} (or the analogous argument \cite[Proposition 2.1]{Lin25} presented by Y. Lin) then gives rise to a smooth orientable $X$-bundle $E\to\Sigma_g$ with monodromy diffeomorphisms $a_1,c_1,\dots,a_g,c_g$ whose total space is not spin. Referring to the kernel of the natural homomorphism
\[\pi_0(\operatorname{Diff}^+(X^\circ,\partial))\to\operatorname{Aut}(H_*(X^\circ;\Bbb Z))\]
as the \emph{Torelli subgroup} of the smooth mapping class group of $X^\circ$ rel boundary, work 
\cite[Corollary C(1)]{OrsonPowell25} of Orson and Powell shows that the Torelli subgroup is exactly comprised of elements of the smooth mapping class group of $X^\circ$ rel boundary whose inclusions into the topological mapping class group of $X^\circ$ rel boundary are trivial (so, the nontrivial elements of the Torelli subgroup are exactly the relatively exotic mapping classes of $X^\circ$). If the smooth mapping class (rel boundary) of each diffeomorphism $a_i^\circ,c_i^\circ$ lies in the Torelli subgroup, each diffeomorphism $a_i,c_i$ will act trivially on the second cohomology of $X$, and will thus preserve the isomorphism class of $\mathfrak s$. In addition, the maximal positive-definite subbundle $H^+(E)$ of the second cohomology bundle of $E\to\Sigma_g$ will be trivial, and will thus certainly have vanishing second Stiefel-Whitney class. Lemma \ref{lem:BKlem} therefore shows that for $X$ as in the lemma, the boundary Dehn twist on $X^\circ$ cannot be represented by any diffeomorphism of the form $[a_1^\circ,c_1^\circ]\cdots[a_g^\circ,c_g^\circ]$, where each $a_i^\circ,c_i^\circ$ represents a member of the Torelli subgroup of the smooth mapping class group of $X^\circ$ rel boundary. This gives Corollary \ref{cor:ourBKlem}:

\begin{corollary}\label{cor:ourBKlem}
For $X$ as in Lemma \ref{lem:BKlem}, the boundary Dehn twist on $X^\circ$ has nontrivial image in the abelianization of the Torelli subgroup of the smooth mapping class group of $X^\circ$ rel boundary. 
\end{corollary}

\section*{Acknowledgements}
\label{sec:thanks}

The author would like to thank Aaron Landesman, Yujie Lin, Tom Mrowka, Anubhav Mukherjee, and Mary Stelow for helpful discussions. The author would also like to thank the School of Science, the Department of Mathematics, and the Office of Graduate Education for support through the MIT Dean of Science fellowship during their doctoral studies, as well as NSF grant DMS-2105512 and the Simons Foundation Award \#994330 (Simons Collaboration on New Structures in Low-Dimensional Topology).

\bibliography{Bibliography.bib}

@article {Atiyah58,
    AUTHOR = {Atiyah, M. F.},
     TITLE = {On analytic surfaces with double points},
   JOURNAL = {Proc. Roy. Soc. London Ser. A},
  FJOURNAL = {Proceedings of the Royal Society. London. Series A.
              Mathematical, Physical and Engineering Sciences},
    VOLUME = {247},
      YEAR = {1958},
     PAGES = {237--244},
      ISSN = {0962-8444,2053-9169},
   MRCLASS = {32.00},
  MRNUMBER = {95974},
MRREVIEWER = {E.\ Calabi},
       DOI = {10.1098/rspa.1958.0181},
       URL = {https://doi.org/10.1098/rspa.1958.0181},
}

@book {AtiyahMacdonald69,
    AUTHOR = {Atiyah, M. F. and Macdonald, I. G.},
     TITLE = {Introduction to commutative algebra},
 PUBLISHER = {Addison-Wesley Publishing Co., Reading, Mass.-London-Don
              Mills, Ont.},
      YEAR = {1969},
     PAGES = {ix+128},
   MRCLASS = {13.00},
  MRNUMBER = {242802},
MRREVIEWER = {Johnny\ A.\ Johnson},
}

@article {Baraglia19,
    AUTHOR = {Baraglia, David},
     TITLE = {Obstructions to smooth group actions on 4-manifolds from
              families {S}eiberg-{W}itten theory},
   JOURNAL = {Adv. Math.},
  FJOURNAL = {Advances in Mathematics},
    VOLUME = {354},
      YEAR = {2019},
     PAGES = {106730, 32},
      ISSN = {0001-8708,1090-2082},
   MRCLASS = {57R57 (57M60 57R50)},
  MRNUMBER = {3981995},
MRREVIEWER = {Fortun\'e\ Massamba},
       DOI = {10.1016/j.aim.2019.106730},
       URL = {https://doi.org/10.1016/j.aim.2019.106730},
}

@article {BaragliaKonno22,
    AUTHOR = {Baraglia, David and Konno, Hokuto},
     TITLE = {On the {B}auer-{F}uruta and {S}eiberg-{W}itten invariants of
              families of 4-manifolds},
   JOURNAL = {J. Topol.},
  FJOURNAL = {Journal of Topology},
    VOLUME = {15},
      YEAR = {2022},
    NUMBER = {2},
     PAGES = {505--586},
      ISSN = {1753-8416,1753-8424},
   MRCLASS = {57K41 (57R22 57R57)},
  MRNUMBER = {4441598},
       DOI = {10.1112/topo.12229},
       URL = {https://doi.org/10.1112/topo.12229},
}

@article {BaragliaKonno23,
    AUTHOR = {Baraglia, David and Konno, Hokuto},
     TITLE = {A note on the {N}ielsen realization problem for {$K3$}
              surfaces},
   JOURNAL = {Proc. Amer. Math. Soc.},
  FJOURNAL = {Proceedings of the American Mathematical Society},
    VOLUME = {151},
      YEAR = {2023},
    NUMBER = {9},
     PAGES = {4079--4087},
      ISSN = {0002-9939,1088-6826},
   MRCLASS = {57R50 (57M60)},
  MRNUMBER = {4607650},
MRREVIEWER = {Wolfgang\ H.\ Heil},
       DOI = {10.1090/proc/15544},
       URL = {https://doi.org/10.1090/proc/15544},
}

@misc{BaragliaKonno25,
      title={Irreducible 4-manifolds can admit exotic diffeomorphisms}, 
      author={David Baraglia and Hokuto Konno},
      year={2025},
      eprint={2412.14398},
      archivePrefix={arXiv},
      primaryClass={math.GT},
      url={https://arxiv.org/abs/2412.14398}, 
}

@article {Birman70,
    AUTHOR = {Birman, Joan S.},
     TITLE = {Abelian quotients of the mapping class group of a
              {$2$}-manifold},
   JOURNAL = {Bull. Amer. Math. Soc.},
  FJOURNAL = {Bulletin of the American Mathematical Society},
    VOLUME = {76},
      YEAR = {1970},
     PAGES = {147--150},
      ISSN = {0002-9904},
   MRCLASS = {30.45},
  MRNUMBER = {249603},
MRREVIEWER = {B.\ Maskit},
       DOI = {10.1090/S0002-9904-1970-12406-5},
       URL = {https://doi.org/10.1090/S0002-9904-1970-12406-5},
}

@incollection {Ehresmann51,
    AUTHOR = {Ehresmann, Charles},
     TITLE = {Les connexions infinit\'esimales dans un espace fibr\'e{}
              diff\'erentiable},
 BOOKTITLE = {Colloque de topologie (espaces fibr\'es), {B}ruxelles, 1950},
     PAGES = {29--55},
 PUBLISHER = {Georges Thone, Li\`ege},
      YEAR = {1951},
   MRCLASS = {53.0X},
  MRNUMBER = {42768},
MRREVIEWER = {H.\ Samelson},
}

@article {GalatiusRandal-Williams16,
    AUTHOR = {Galatius, S{\o}ren and Randal-Williams, Oscar},
     TITLE = {Abelian quotients of mapping class groups of highly connected manifolds},
   JOURNAL = {Math. Ann.},
  FJOURNAL = {Mathematische Annalen},
    VOLUME = {365},
      YEAR = {2016},
    NUMBER = {1-2},
     PAGES = {857--879},
      ISSN = {0025-5831,1432-1807},
   MRCLASS = {55N22 (57R15 57R50 57R90)},
  MRNUMBER = {3498929},
MRREVIEWER = {Semen\ S.\ Podkorytov},
       DOI = {10.1007/s00208-015-1300-2},
       URL = {https://doi.org/10.1007/s00208-015-1300-2},
}

@article {GalatiusRandal-Williams14,
    AUTHOR = {Galatius, S{\o}ren and Randal-Williams, Oscar},
     TITLE = {Stable moduli spaces of high-dimensional manifolds},
   JOURNAL = {Acta Math.},
  FJOURNAL = {Acta Mathematica},
    VOLUME = {212},
      YEAR = {2014},
    NUMBER = {2},
     PAGES = {257--377},
      ISSN = {0001-5962,1871-2509},
   MRCLASS = {55R40 (55Pxx 57Mxx)},
  MRNUMBER = {3207759},
MRREVIEWER = {Semen\ S.\ Podkorytov},
       DOI = {10.1007/s11511-014-0112-7},
       URL = {https://doi.org/10.1007/s11511-014-0112-7},
}

@article {GalatiusRandal-Williams17,
    AUTHOR = {Galatius, S{\o}ren and Randal-Williams, Oscar},
     TITLE = {Homological stability for moduli spaces of high dimensional
              manifolds. {II}},
   JOURNAL = {Ann. of Math. (2)},
  FJOURNAL = {Annals of Mathematics. Second Series},
    VOLUME = {186},
      YEAR = {2017},
    NUMBER = {1},
     PAGES = {127--204},
      ISSN = {0003-486X,1939-8980},
   MRCLASS = {57R90 (55P47 57R15 57R56)},
  MRNUMBER = {3665002},
MRREVIEWER = {Sam\ Nariman},
       DOI = {10.4007/annals.2017.186.1.4},
       URL = {https://doi.org/10.4007/annals.2017.186.1.4},
}

@book {GompfStipsicz99,
    AUTHOR = {Gompf, Robert E. and Stipsicz, Andr\'as I.},
     TITLE = {{$4$}-manifolds and {K}irby calculus},
    SERIES = {Graduate Studies in Mathematics},
    VOLUME = {20},
 PUBLISHER = {American Mathematical Society, Providence, RI},
      YEAR = {1999},
     PAGES = {xvi+558},
      ISBN = {0-8218-0994-6},
   MRCLASS = {57N13 (14J80 32Q55 57-02 57R17 57R57 57R65)},
  MRNUMBER = {1707327},
MRREVIEWER = {Nikolai\ N.\ Saveliev},
       DOI = {10.1090/gsm/020},
       URL = {https://doi.org/10.1090/gsm/020},
}

@article {Harer85,
    AUTHOR = {Harer, John L.},
     TITLE = {Stability of the homology of the mapping class groups of
              orientable surfaces},
   JOURNAL = {Ann. of Math. (2)},
  FJOURNAL = {Annals of Mathematics. Second Series},
    VOLUME = {121},
      YEAR = {1985},
    NUMBER = {2},
     PAGES = {215--249},
      ISSN = {0003-486X,1939-8980},
   MRCLASS = {57M99 (20F34)},
  MRNUMBER = {786348},
MRREVIEWER = {K.\ Vogtmann},
       DOI = {10.2307/1971172},
       URL = {https://doi.org/10.2307/1971172},
}

@book {Hartshorne77,
    AUTHOR = {Hartshorne, Robin},
     TITLE = {Algebraic geometry},
    SERIES = {Graduate Texts in Mathematics},
    VOLUME = {No. 52},
 PUBLISHER = {Springer-Verlag, New York-Heidelberg},
      YEAR = {1977},
     PAGES = {xvi+496},
      ISBN = {0-387-90244-9},
   MRCLASS = {14-01},
  MRNUMBER = {463157},
MRREVIEWER = {Robert\ Speiser},
}

@article {KangParkTaniguchi26,
    AUTHOR = {Kang, Sungkyung and Park, JungHwan and Taniguchi, Masaki},
     TITLE = {Exotic {D}ehn twists and homotopy coherent group actions},
   JOURNAL = {Invent. Math.},
  FJOURNAL = {Inventiones Mathematicae},
    VOLUME = {243},
      YEAR = {2026},
    NUMBER = {1},
     PAGES = {209--241},
      ISSN = {0020-9910,1432-1297},
   MRCLASS = {99-06},
  MRNUMBER = {5008150},
       DOI = {10.1007/s00222-025-01378-1},
       URL = {https://doi.org/10.1007/s00222-025-01378-1},
}

@article {Konno21,
    AUTHOR = {Konno, Hokuto},
     TITLE = {Characteristic classes via 4-dimensional gauge theory},
   JOURNAL = {Geom. Topol.},
  FJOURNAL = {Geometry \& Topology},
    VOLUME = {25},
      YEAR = {2021},
    NUMBER = {2},
     PAGES = {711--773},
      ISSN = {1465-3060,1364-0380},
   MRCLASS = {57K41 (55R40 57R57)},
  MRNUMBER = {4251435},
       DOI = {10.2140/gt.2021.25.711},
       URL = {https://doi.org/10.2140/gt.2021.25.711},
}

@misc{KonnoLin23,
      title={Homological instability for moduli spaces of smooth 4-manifolds}, 
      author={Hokuto Konno and Jianfeng Lin},
      year={2023},
      eprint={2211.03043},
      archivePrefix={arXiv},
      primaryClass={math.GT},
      url={https://arxiv.org/abs/2211.03043}, 
}

@misc{KonnoMallickTaniguchi24,
      title={Exotic {D}ehn twists on 4-manifolds}, 
      author={Hokuto Konno and Abhishek Mallick and Masaki Taniguchi},
      year={2024},
      eprint={2306.08607},
      archivePrefix={arXiv},
      primaryClass={math.GT},
      url={https://arxiv.org/abs/2306.08607}, 
}

@misc{Konno...24a,
      title={On four-dimensional {D}ehn twists and {M}ilnor fibrations}, 
      author={Hokuto Konno and Jianfeng Lin and Anubhav Mukherjee and Juan Mu{\~n}oz-Ech{\'a}niz},
      year={2024},
      eprint={2409.11961},
      archivePrefix={arXiv},
      primaryClass={math.GT},
      url={https://arxiv.org/abs/2409.11961}, 
}

@misc{Konno...24b,
      title={The monodromy diffeomorphism of weighted singularities and {S}eiberg--{W}itten theory}, 
      author={Hokuto Konno and Jianfeng Lin and Anubhav Mukherjee and Juan Mu{\~n}oz-Ech{\'a}niz},
      year={2024},
      eprint={2411.12202},
      archivePrefix={arXiv},
      primaryClass={math.GT},
      url={https://arxiv.org/abs/2411.12202}, 
}

@article {Korkmaz02,
    AUTHOR = {Korkmaz, Mustafa},
     TITLE = {Low-dimensional homology groups of mapping class groups: a
              survey},
   JOURNAL = {Turkish J. Math.},
  FJOURNAL = {Turkish Journal of Mathematics},
    VOLUME = {26},
      YEAR = {2002},
    NUMBER = {1},
     PAGES = {101--114},
      ISSN = {1300-0098,1303-6149},
   MRCLASS = {57M05 (20J05 57M07 57N05)},
  MRNUMBER = {1892804},
MRREVIEWER = {Michel\ Coornaert},
}

@article {Krannich20,
    AUTHOR = {Krannich, Manuel},
     TITLE = {Mapping class groups of highly connected {$(4k+2)$}-manifolds},
   JOURNAL = {Selecta Math. (N.S.)},
  FJOURNAL = {Selecta Mathematica. New Series},
    VOLUME = {26},
      YEAR = {2020},
    NUMBER = {5},
     PAGES = {Paper No. 81, 49},
      ISSN = {1022-1824,1420-9020},
   MRCLASS = {57R50 (55N22 57R60)},
  MRNUMBER = {4182837},
MRREVIEWER = {Bena\ Tshishiku},
       DOI = {10.1007/s00029-020-00600-7},
       URL = {https://doi.org/10.1007/s00029-020-00600-7},
}

@article {KrannichKupers25,
    AUTHOR = {Krannich, Manuel and Kupers, Alexander},
     TITLE = {On {T}orelli groups and {D}ehn twists of smooth 4-manifolds},
   JOURNAL = {Bull. Lond. Math. Soc.},
  FJOURNAL = {Bulletin of the London Mathematical Society},
    VOLUME = {57},
      YEAR = {2025},
    NUMBER = {3},
     PAGES = {956--963},
      ISSN = {0024-6093,1469-2120},
   MRCLASS = {57K40 (57N37 57R52 57S05)},
  MRNUMBER = {4877539},
       DOI = {10.1112/blms.70009},
       URL = {https://doi.org/10.1112/blms.70009},
}

@article {KreckSu25,
    AUTHOR = {Kreck, Matthias and Su, Yang},
     TITLE = {Mapping class group of manifolds which look like 3-dimensional
              complete intersections},
   JOURNAL = {Duke Math. J.},
  FJOURNAL = {Duke Mathematical Journal},
    VOLUME = {174},
      YEAR = {2025},
    NUMBER = {3},
     PAGES = {501--574},
      ISSN = {0012-7094,1547-7398},
   MRCLASS = {57R65 (14J30 14J50 57K50 57R19)},
  MRNUMBER = {4888148},
MRREVIEWER = {Michael\ Wiemeler},
       DOI = {10.1215/00127094-2024-0036},
       URL = {https://doi.org/10.1215/00127094-2024-0036},
}

@article {KronheimerMrowka20,
    AUTHOR = {Kronheimer, P. B. and Mrowka, T. S.},
     TITLE = {The {D}ehn twist on a sum of two {$K3$} surfaces},
   JOURNAL = {Math. Res. Lett.},
  FJOURNAL = {Mathematical Research Letters},
    VOLUME = {27},
      YEAR = {2020},
    NUMBER = {6},
     PAGES = {1767--1783},
      ISSN = {1073-2780,1945-001X},
   MRCLASS = {57R50 (57K40)},
  MRNUMBER = {4216604},
       DOI = {10.4310/MRL.2020.v27.n6.a8},
       URL = {https://doi.org/10.4310/MRL.2020.v27.n6.a8},
}

@article {Lin23,
    AUTHOR = {Lin, Jianfeng},
     TITLE = {Isotopy of the {D}ehn twist on {$K3\, \#\, K3$} after a single
              stabilization},
   JOURNAL = {Geom. Topol.},
  FJOURNAL = {Geometry \& Topology},
    VOLUME = {27},
      YEAR = {2023},
    NUMBER = {5},
     PAGES = {1987--2012},
      ISSN = {1465-3060,1364-0380},
   MRCLASS = {57R52 (55P91 57R50 57R57)},
  MRNUMBER = {4621924},
MRREVIEWER = {Andrey\ Gogolev},
       DOI = {10.2140/gt.2023.27.1987},
       URL = {https://doi.org/10.2140/gt.2023.27.1987},
}

@misc{Lin25,
      title={A note on the boundary {D}ehn twist of {$K3$} surfaces}, 
      author={Yujie Lin},
      year={2025},
      eprint={2506.10444},
      archivePrefix={arXiv},
      primaryClass={math.GT},
      url={https://arxiv.org/abs/2506.10444}, 
}

@article {MasdenWeiss07,
    AUTHOR = {Madsen, Ib and Weiss, Michael},
     TITLE = {The stable moduli space of {R}iemann surfaces: {M}umford's
              conjecture},
   JOURNAL = {Ann. of Math. (2)},
  FJOURNAL = {Annals of Mathematics. Second Series},
    VOLUME = {165},
      YEAR = {2007},
    NUMBER = {3},
     PAGES = {843--941},
      ISSN = {0003-486X,1939-8980},
   MRCLASS = {14H10 (14F43 19D06 55P47)},
  MRNUMBER = {2335797},
MRREVIEWER = {Ulrike\ Tillmann},
       DOI = {10.4007/annals.2007.165.843},
       URL = {https://doi.org/10.4007/annals.2007.165.843},
}

@article {Miller86,
    AUTHOR = {Miller, Edward Y.},
     TITLE = {The homology of the mapping class group},
   JOURNAL = {J. Differential Geom.},
  FJOURNAL = {Journal of Differential Geometry},
    VOLUME = {24},
      YEAR = {1986},
    NUMBER = {1},
     PAGES = {1--14},
      ISSN = {0022-040X,1945-743X},
   MRCLASS = {32G15 (57N05)},
  MRNUMBER = {857372},
MRREVIEWER = {Ronnie\ Lee},
       URL = {http://projecteuclid.org/euclid.jdg/1214440254},
}

@book {Milnor63,
    AUTHOR = {Milnor, J.},
     TITLE = {Morse theory},
    SERIES = {Annals of Mathematics Studies},
    VOLUME = {No. 51},
      NOTE = {Based on lecture notes by M. Spivak and R. Wells},
 PUBLISHER = {Princeton University Press, Princeton, NJ},
      YEAR = {1963},
     PAGES = {vi+153},
   MRCLASS = {57.50 (53.72)},
  MRNUMBER = {163331},
MRREVIEWER = {H.\ I.\ Levine},
}

@article {Morita87,
    AUTHOR = {Morita, Shigeyuki},
     TITLE = {Characteristic classes of surface bundles},
   JOURNAL = {Invent. Math.},
  FJOURNAL = {Inventiones Mathematicae},
    VOLUME = {90},
      YEAR = {1987},
    NUMBER = {3},
     PAGES = {551--577},
      ISSN = {0020-9910,1432-1297},
   MRCLASS = {57R20 (32G15 57M99 58D15)},
  MRNUMBER = {914849},
MRREVIEWER = {Ronnie\ Lee},
       DOI = {10.1007/BF01389178},
       URL = {https://doi.org/10.1007/BF01389178},
}

@article {Mumford67,
    AUTHOR = {Mumford, David},
     TITLE = {Abelian quotients of the {T}eichm\"uller modular group},
   JOURNAL = {J. Analyse Math.},
  FJOURNAL = {Journal d'Analyse Math\'ematique},
    VOLUME = {18},
      YEAR = {1967},
     PAGES = {227--244},
      ISSN = {0021-7670,1565-8538},
   MRCLASS = {14.51 (32.00)},
  MRNUMBER = {219543},
MRREVIEWER = {R.\ C.\ Gunning},
       DOI = {10.1007/BF02798046},
       URL = {https://doi.org/10.1007/BF02798046},
}

@incollection {Mumford83,
    AUTHOR = {Mumford, David},
     TITLE = {Towards an enumerative geometry of the moduli space of curves},
 BOOKTITLE = {Arithmetic and geometry, {V}ol. {II}},
    SERIES = {Progr. Math.},
    VOLUME = {36},
     PAGES = {271--328},
 PUBLISHER = {Birkh\"auser Boston, Boston, MA},
      YEAR = {1983},
      ISBN = {3-7643-3133-X},
   MRCLASS = {14H10 (14C15)},
  MRNUMBER = {717614},
MRREVIEWER = {Werner\ Kleinert},
}

@article {OrsonPowell25,
    AUTHOR = {Orson, Patrick and Powell, Mark},
     TITLE = {Mapping class groups of simply connected 4-manifolds with
              boundary},
   JOURNAL = {J. Differential Geom.},
  FJOURNAL = {Journal of Differential Geometry},
    VOLUME = {131},
      YEAR = {2025},
    NUMBER = {1},
     PAGES = {199--275},
      ISSN = {0022-040X,1945-743X},
   MRCLASS = {57K40 (57N37 57R52 57S05)},
  MRNUMBER = {4947553},
       DOI = {10.4310/jdg/1755544135},
       URL = {https://doi.org/10.4310/jdg/1755544135},
}

@article {Powell78,
    AUTHOR = {Powell, Jerome},
     TITLE = {Two theorems on the mapping class group of a surface},
   JOURNAL = {Proc. Amer. Math. Soc.},
  FJOURNAL = {Proceedings of the American Mathematical Society},
    VOLUME = {68},
      YEAR = {1978},
    NUMBER = {3},
     PAGES = {347--350},
      ISSN = {0002-9939,1088-6826},
   MRCLASS = {57A05 (30A46)},
  MRNUMBER = {494115},
MRREVIEWER = {Hugh\ M.\ Hilden},
       DOI = {10.2307/2043120},
       URL = {https://doi.org/10.2307/2043120},
}

@misc{Tilton25,
      title={The Boundary {D}ehn Twist on a Punctured Connected Sum of Two {$K3$} Surfaces is Nontrivial in the Smooth Mapping Class Group}, 
      author={Scotty Tilton},
      year={2025},
      eprint={arXiv:2511.16804},
      archivePrefix={arXiv},
      primaryClass={math.AT},
      url={https://arxiv.org/abs/2511.16804}, 
}

\end{document}